\documentclass[11pt,letterpaper,reqno]{amsart}

\usepackage{amsmath,amssymb,amsthm,bbm,bm}
\usepackage{color}
\usepackage[colorlinks=true, allcolors=blue]{hyperref}
\usepackage{mathrsfs}
\usepackage{mathtools}

\usepackage[noabbrev,capitalize,nameinlink]{cleveref}
\usepackage[noadjust]{cite}

\usepackage{fullpage}

\usepackage{skull}

\usepackage{graphics}
\usepackage{pifont}
\usepackage[T1]{fontenc}
\usepackage{tikz}
\usetikzlibrary{arrows.meta}
\usepackage{environ}
\usepackage{framed}
\usepackage{url}
\usepackage[labelfont=bf]{caption}
\usepackage{cite}
\usepackage{framed}
\usepackage[framemethod=tikz]{mdframed}
\usepackage{appendix}
\usepackage{graphicx}
\usepackage{tcolorbox}
\allowdisplaybreaks[1]
\usepackage{enumerate}
\usepackage[shortlabels]{enumitem}
\crefformat{enumi}{#2#1#3}
\crefrangeformat{enumi}{#3#1#4 to~#5#2#6}
\crefmultiformat{enumi}{#2#1#3}
{ and~#2#1#3}{, #2#1#3}{, and~#2#1#3}

\crefname{algocf}{Algorithm}{Algorithms}

\crefname{equation}{}{} 
\AtBeginEnvironment{appendices}{\crefalias{section}{appendix}} 

\crefname{algocf}{Algorithm}{Algorithms}

\numberwithin{equation}{section}
\newtheorem{theorem}{Theorem}[section]
\newtheorem{proposition}[theorem]{Proposition}
\newtheorem{lemma}[theorem]{Lemma}
\newtheorem{claim}[theorem]{Claim}
\crefname{claim}{Claim}{Claims}

\newtheorem{corollary}[theorem]{Corollary}
\newtheorem{conjecture}[theorem]{Conjecture}
\newtheorem*{question*}{Question}

\theoremstyle{definition}
\newtheorem{definition}[theorem]{Definition}

\newtheorem*{definition*}{Definition}

\newtheorem{setup}[theorem]{Setup}
\newtheorem{algorithm}[theorem]{Algorithm}

\theoremstyle{remark}
\newtheorem*{remark}{Remark}
\newtheorem*{notation*}{Notation}
\newtheorem{notation}[theorem]{Notation}

\newcommand{\mb}{\mathbb}
\newcommand{\mc}{\mathcal}
\newcommand{\cS}{\mathcal{S}}

\renewcommand{\epsilon}{\varepsilon}
\newcommand{\E}{\mathop{\mathbb{E}}}
\DeclareMathOperator{\ex}{ex}

\newcommand{\skel}[1]{{d_1({#1})}}

\allowdisplaybreaks

\title{Ramsey and Tur\'an numbers of sparse hypergraphs}

\author[Fox]{Jacob Fox}
\author[Sankar]{Maya Sankar}
\author[Simkin]{Michael Simkin}
\address{Department of Mathematics, Massachusetts Institute of Technology}
\email{msimkin@mit.edu}
\author[Tidor]{Jonathan Tidor}
\author[Zhou]{Yunkun Zhou}
\address{Department of Mathematics, Stanford University}
\email{\{jacobfox, mayars, jtidor, yunkunzhou\}@stanford.edu}
\thanks{Fox was supported a Packard Fellowship and by NSF Awards DMS-1800053 and DMS-2154169. Sankar was supported by a Fannie and John Hertz Foundation Fellowship and NSF Graduate Research Fellowship DGE-1656518. Simkin was partially supported by the Center of Mathematical Sciences and Applications at Harvard University. Tidor was supported by a Stanford Science Fellowship. Zhou was supported by NSF Graduate Research Fellowship DGE-1656518.}

\begin{document}

\begin{abstract}
Degeneracy plays an important role in understanding Tur\'an- and Ramsey-type properties of graphs. Unfortunately, the usual hypergraphical generalization of degeneracy fails to capture these properties. We define the skeletal degeneracy of a $k$-uniform hypergraph as the degeneracy of its $1$-skeleton (i.e., the graph formed by replacing every $k$-edge by a $k$-clique). We prove that skeletal degeneracy controls hypergraph Tur\'an and Ramsey numbers in a similar manner to (graphical) degeneracy.

Specifically, we show that $k$-uniform hypergraphs with bounded skeletal degeneracy have linear Ramsey number. This is the hypergraph analogue of the Burr--Erd\H{o}s conjecture (proved by Lee). In addition, we give upper and lower bounds of the same shape for the Tur\'an number of a $k$-uniform $k$-partite hypergraph in terms of its skeletal degeneracy. The proofs of both results use the technique of dependent random choice. In addition, the proof of our Ramsey result uses the `random greedy process' introduced by Lee in his resolution of the Burr--Erd\H{o}s conjecture.
\end{abstract}

\maketitle

\section{Introduction}

Two central problems in extremal combinatorics are to better understand the Tur\'an and Ramsey numbers of graphs and hypergraphs. To recap the definitions, let $H$ be a $k$-uniform hypergraph and $n,q$ be positive integers. The \textit{Tur\'an number} $\ex(n,H)$ is the maximum number of edges in a $k$-uniform hypergraph on $n$ vertices that is $H$-free. The \textit{Ramsey number} $r(H;q)$ is the minimum positive integer $N$ for which every coloring of the edges of the complete $k$-uniform hypergraph $K_N^{(k)}$ using $q$ colors contains a monochromatic copy of $H$.

Estimating Ramsey and Tur\'an numbers has remained a challenging problem (see, for example, the surveys \cite{Kev11,CFS15,MS20} and the book \cite{GRS90}). It is therefore natural to seek parameters of the hypergraph $H$ that control these numbers. As we will explain, for $k=2$ (i.e., the graph case) the notion of degeneracy has been central in this regard. The purpose of this paper is to propose an analogue for hypergraphs, which we term \textit{skeletal degeneracy}, and to prove that it emulates the role of degeneracy for graphs, at least as it concerns Tur\'an and Ramsey numbers.

Among the earliest results in modern combinatorics are Tur\'an's theorem and Ramsey's theorem \cite{Tur41,Ram29}. Tur\'an's theorem determines $\ex(n,K_r)$ exactly, showing that among $K_r$-free graphs on $n$ vertices the complete balanced $(r-1)$-partite graph has the most edges (the case $r=3$ is due to Mantel \cite{Man07} and goes back more than a century). Ramsey's theorem guarantees that Ramsey numbers exist. Besides their inherent interest, Tur\'an and Ramsey problems have served as strong motivations to develop new methods and techniques. For instance, the early quantitative upper bound $r(K_r;2) \leq \binom{2r-2}{r-1} < 4^r$ was proved by Erd\H{o}s and Szekeres \cite{ESz35}. Despite much effort, the exponent $4$ in the upper bound resisted improvement for many decades until the very recent breakthrough of Campos, Griffiths, Morris, and Sahasrabudhe \cite{CGMS23}. 

Additionally, an early and extremely influential use of the probabilistic method in combinatorics was Erd\H{o}s's lower bound on $r(K_r;2)$ leveraging the random graph $\mb G(n;1/2)$ \cite{Erd47}. To this day, newcomers to the field often see this example before any other.  In 1975, Spencer \cite{Spe75} slightly improved the lower bound, in a paper that introduced a common form of the Lov\'asz local lemma, but this remains the state of the art.

With respect to Tur\'an numbers, a consequence of the Erd\H{o}s--Stone--Simonovits theorem \cite{ESi66,ESt46} is that for a fixed graph $H$ with chromatic number $\chi = \chi(H)$ the extremal number satisfies $\ex(n,H) = (1 - 1/(\chi-1) + o(1))n^2\!/2$. Thus, for non-bipartite $H$, the asymptotics of $\ex(n,H)$ are determined exactly by $\chi(H)$. 

In contrast to the non-bipartite case, for bipartite $H$ the Erd\H{o}s--Stone--Simonovits theorem only says that $\ex(n,H)=o(n^2)$, and these numbers are far less understood. In fact, even determining the order of magnitude of $\ex(n,H)$ is open in general. In the search for graphical parameters of $H$ that control $\ex(n,H)$, a natural candidate is the \textit{degeneracy} of $H$, the smallest $d$ such that every subgraph of $H$ has minimum degree at most $d$. A longstanding conjecture of Erd\H{o}s~\cite{Erd67} states that if $H$ is a $d$-degenerate bipartite graph then the Tur\'an number $\ex(n,H)$ has order at most $n^{2-1/d}$. Despite much interest (see, for example, \cite{FS09,CJL21,CL21,Jan19,Jan20,Jan21,JS22,BJST23,ST20}), the conjecture remains open. Currently, the best lower bound is $\ex(n,H)\geq \Omega(n^{2-2/d})$, achieved by a straightforward probabilistic construction. The best known upper bound is due to Alon, Krivelevich, and Sudakov \cite{AKS03}. In an early application of dependent random choice (which also figures prominently in this paper) they proved the following theorem.

\begin{theorem}[{\cite[Theorem 3.5]{AKS03}}]
\label{thm:AKS-graph-turan}
For every $d$-degenerate bipartite graph $H$,
\[\ex(n,H)\leq O_H\left(n^{2-\frac1{4d}}\right).\]
\end{theorem}

In addition to controlling the Tur\'an exponent, degeneracy plays an important role in graph Ramsey theory. In 1973, Burr and Erd\H{o}s \cite{BE75} conjectured that graphs with bounded degeneracy have Ramsey number linear in the number of vertices. They also made the related, weaker, conjecture that bounded \textit{degree} graphs have linear Ramsey number. These conjectures garnered much interest (see, for instance, \cite{GRR00,GRR01,CFS12} for the bounded degree case and \cite{KR04,KS03,FS09} for the bounded degeneracy case). Although the bounded degree case was resolved by Chvat\'al, R\"odl, Szemer\'edi, and Trotter \cite{CRST83} in 1983, it would take until 2017 until the bounded degeneracy conjecture was proved by Lee \cite{Lee17}.

\begin{theorem}[{\cite[Theorem 1.1]{Lee17}}]
\label{thm:Lee-graph-ramsey}
For every $d$-degenerate $n$-vertex graph $H$,
\[r(H;2)\leq O_d(n).\]
\end{theorem}

In this paper, we seek hypergraph analogues of \cref{thm:AKS-graph-turan,thm:Lee-graph-ramsey}. A na\"ive attempt would be to use the standard definition of hypergraph degeneracy: the \emph{degeneracy} of a hypergraph $H$ is the least $d$ such that every subhypergraph of $H$ has minimum degree at most $d$. This notion was studied by Kostochka and R\"odl \cite{KR06}, who showed that the natural analogue of the Burr--Erd\H{o}s conjecture fails for hypergraphs. Specifically, they described a family of 1-degenerate 4-uniform hypergraphs $H$ on $n$ vertices satisfying $r(H;2)\geq 2^{\Omega(n^{1/3})}$, as well as more complicated examples showing that the conjecture fails for 3-uniform hypergraphs. (See also \cite{CFR17} for more examples of this phenomenon.)

It is also straightforward to show that the hypergraph analogue of Erd\H{o}s's conjecture regarding Tur\'an numbers is false. Indeed, for every $\epsilon>0$ there is a 1-degenerate $k$-uniform $k$-partite hypergraph $H$ with $\ex(n,H)>n^{k-\epsilon}$. We give the details of this simple construction in \cref{sec:turan-lower} (see \cref{cor:degen-turan-counterex}).

To overcome the issues above we introduce a new notion of degeneracy for hypergraphs that we call \textit{skeletal degeneracy}.

\begin{definition}\label{defn:skeletal degeneracy}
Let $H$ be a $k$-uniform hypergraph. For $0\leq i<k$, its \emph{$i$-skeleton} $H^{(i)}$ is the $(i+1)$-uniform hypergraph on the same vertex set with all $(i+1)$-edges that are contained in some $k$-edge of $H$. The \emph{$i$-th skeletal degeneracy} of $H$, denoted $d_i(H)$, is the degeneracy of $H^{(i)}$, i.e., the least $d$ such that every subhypergraph of $H^{(i)}$ has minimum degree at most $d$. Equivalently, this is the smallest $d$ for which there exists an ordering $v_1,\ldots,v_n$ of $V(H)$ in which each vertex $v_j$ is the highest-numbered vertex of at most $d$ of the $(i+1)$-edges of $H^{(i)}$. We refer to $d_1(H)$ simply as the \emph{skeletal degeneracy} of $H$. 
\end{definition}

For a $k$-uniform hypergraph $H$ the $(k-1)$-st skeletal degeneracy $d_{k-1}(H)$ agrees with the standard definition of degeneracy given above. Moreover, although $d_{k-1}(H)$ is not a good predictor of the Tur\'an and Ramsey properties of $H$, we will show that the skeletal degeneracy $d_1(H)$ is.

Observe that the skeletal degeneracy can equivalently be defined as the degeneracy of the graph formed by replacing every $k$-edge of a hypergraph $H$ by a $k$-clique. Although it may seem that this definition does not capture the hypergraph structure of $G$, we show that it is nevertheless sufficient to control the Tur\'an and Ramsey properties of $H$. Our first of two main results shows that this is the case for Tur\'an numbers.

\begin{theorem}
\label{thm:k-unif-turan-main}
For every $k \geq 2$ there exist constants $C_k,c_k>0$ such that the following holds. For any $k$-uniform $k$-partite hypergraph $H$ with at least two edges, the Tur\'an number of $H$ satisfies
\[\Omega_k\left(n^{k-\frac{C_k}{d_1(H)}}\right) \leq \ex(n,H) \leq O_H\left(n^{k-\frac{c_k}{d_1(H)^{k-1}}}\right).\]
\end{theorem}

Note that for $k=2$, the exponent in the upper and lower bounds are both of the shape ${2-c/d_1(H)}$. However, as $k$ increases, the gap between the upper and lower bounds in this theorem increases.  In \cref{sec:turan-lower} we give examples showing that (up to the values of the constants $C_k,c_k$) both growth rates are possible. We discuss this more and conjecture more precise bounds in \cref{sec:open}.

We remark that \cref{thm:k-unif-turan-main} focuses on the $k$-partite case because the asymptotic order of magnitude in the non-partite case is easy to ascertain. Indeed, if $H$ is not $k$-partite, then $\ex(n,H) = \Theta_H(n^k)$; the upper bound is trivial and the lower bound is witnessed by the complete balanced $k$-partite hypergraph. This is not to say that the non-partite case is fully understood. On the contrary, even determining the leading asymptotics of $\ex(n,K_4^{(3)})$ (i.e., tetrahedra-free $3$-uniform hypergraphs) is a notable open problem. Nevertheless, we restrict our attention to the $k$-partite case, where even the order of magnitude is unknown.

Our second main result concerns Ramsey numbers. Conlon, Sudakov, and the first author asked if the hypergraph analogue of the Burr--Erd\H{o}s conjecture holds for skeletal degeneracy \cite[Problem 1.26]{AIM15}. We answer this question in the affirmative.

\begin{theorem}
\label{thm:lin-ramsey-main}
Let $H$ be a $k$-uniform hypergraph on $n$ vertices with skeletal degeneracy $\skel{H}=d$. Then the $q$-color Ramsey number of $H$ satisfies \[r(H;q)\leq O_{k,d,q}(n).\]
\end{theorem}

One nice application of this theorem is to triangulations of surfaces. Let $\Sigma$ be a surface of genus $g$ and let $H$ be a 3-uniform $n$-vertex hypergraph which triangulates $\Sigma$. It can be deduced from Euler's formula (see, e.g.,~\cite[Theorem 4.4.4]{MT01}) that $d_1(H)\leq O(\sqrt{g})$. Thus by \cref{thm:lin-ramsey-main} we see that $r(H;q)\leq O_{g,q}(n)$ where the hidden constant only depends on the genus and the number of colors. To the best of our knowledge, this family of hypergraphs was not previously known to have linear Ramsey number. One can also show that the dependence on genus in the constant in front of the linear term is necessary due to an example based on embedding a large clique into a surface of large genus.

The remainder of this paper is organized as follows. We prove the upper and lower bounds on \cref{thm:k-unif-turan-main} in \cref{sec:turan-upper,sec:turan-lower}, respectively. Using the same techniques, we give a short proof of an almost-linear bound on the Ramsey number in \cref{sec:ramsey-almost-lin}. In \cref{sec:ramsey-lin}, we prove \cref{thm:lin-ramsey-main}. Finally, in \cref{sec:open} we give some open problems.

\vspace{1em}
\noindent\textbf{Acknowledgements:} We thank David Conlon for helpful discussions early in this project's development.

\section{Upper bounds on the Tur\'an number}
\label{sec:turan-upper}

In this section we prove the upper bound in \cref{thm:k-unif-turan-main}. We show how to find a $k$-uniform $k$-partite hypergraph $H$ in any $k$-uniform $k$-partite host hypergraph $G$ with sufficiently many edges. We do this by first pruning $G$ so that all small sets of $(k-1)$-edges in $G$ have many mutual extensions to $k$-edges in $G$. We then show how this extending property suffices to embed $H$ in $G$.

\begin{notation*}
For a $k$-partite $k$-uniform hypergraph $G$ with parts $V_1, \dots, V_k$, write $V_{-t} =  \bigcup_{i\ne t}V_i$. For $S\subseteq V(G)$, we use the notation $G^{(k-2)}[S]$ for the subhypergraph of the $(k-2)$-skeleton of $G$ induced on the vertices $S$. For example, $G^{(k-2)}[V_{-t}]$ is the $(k-1)$-partite $(k-1)$-uniform hypergraph on $V_{-t}$ with edge set $\{e\cap V_{-t}:e\in E(G)\}$.
\end{notation*}

\begin{definition}
Let $G$ be a $k$-partite $k$-uniform hypergraph on vertex set $V_1\sqcup V_2\sqcup\cdots\sqcup V_k$. For a set $F\subseteq E(G^{(k-2)})$, we say that a vertex $v\in V(G)$ is a \emph{mutual extension of $F$} if $e\cup\{v\}\in E(G)$ for each $e\in F$. For each $1\leq t\leq k$ and positive integers $a, d$, we say that $G$ is \textit{$(a, d)$-vertex-extending to $V_t$} if for every set $S$ of at most $d$ vertices in $V_{-t}$, the set $E(G^{(k-2)}[S])$ of $(k-1)$-edges induced by $S$ has at least $a$ mutual extensions (in $V_t$).
\end{definition}

\begin{lemma}
\label{thm:k-uniform-pruning}
Let $d \in \mb N$ and $k \geq 3$. There exist a constant $c_k >0$ and a constant $n_0 = n_0(k,d)$ such that the following holds. If $G$ is a $k$-partite $k$-uniform hypergraph on vertex set $V_1\sqcup\cdots\sqcup V_k$ with $n\geq n_0$ vertices in each part and at least $n^{k- c_k / d^{k-1}}$ edges, then it contains a nonempty subhypergraph $G'\subseteq G$ which is $(n^{1/3}, d)$-vertex-extending to each part.
\end{lemma}

This lemma builds on the following result, which uses the technique of dependent random choice.

\begin{lemma}
\label{lem:intermediate-step-quantitative}
Let $d\geq 2$, $k\geq 3, t, h, n, \lambda$ be positive integers with $1\leq t\leq k$, and let $p\in (0, 1/2)$ be a real number. Suppose these parameters also satisfy the following two inequalities:
\[p^{(\lambda+1)d}n\geq 1\qquad\text{and}\qquad n^{-1} \geq \left(\frac{h}{n}\right)^{\lambda d}\binom{kn}{d}2^{\binom{d}{k-1}}.\]
Then, given any $k$-partite $k$-uniform hypergraph $G$ on vertex set $V_1\sqcup\cdots\sqcup V_k$ with $n$ vertices in each part and at least $pn^k$ edges, the following holds. If $G$ is $(h, (\lambda+1)d)$-vertex-extending to $V_r$ for all $r < t$ then there exists a subhypergraph $G'$ of $G$ on the same vertex set with at least $p^{(\lambda+1)d}n^k$ edges that is $(h, d)$-vertex-extending to $V_r$ for all $r \leq t$.
\end{lemma}

\begin{proof}
Choose $u = \lambda d$ vertices uniformly at random from $V_t$ with replacement and let $X_t$ be the resulting set of vertices.  Let $G'$ be the subhypergraph of $G$ containing only those edges $e=\{e_1,\ldots,e_k\}$ for which $(e\setminus\{e_t\})\cup\{x_t\}$ is an edge of $G$ for each $x_t\in X_t$.

For each $(k-1)$-edge $e\in E(G^{(k-2)}[V_{-t}])$, let $d_G(e)$ be the number of vertices (in $V_t$) that are mutual extensions of $\{e\}$. Then the probability that any such $e$ remains in $G'$ is $(d_G(e)/n)^u$. Since the number of edges of $G$ satisfies $e(G)=\sum_{e \in E(G^{(k-2)}[V_{-t}])} d_G(e) \geq pn^{k}$, by convexity
\[\mb E[e(G')]  = \sum_{e \in E(G^{(k-2)}[V_{-t}])}d_G(e) \left(\frac{d_G(e)}{n}\right)^u \geq \frac{e(G)^{u+1}}{n^u e(G^{(k-2)}[V_{-t}])^u}\geq  \frac{e(G)^{u+1}}{n^{ku}}\geq p^{u+1}n^{k}.\]
    
For a set of edges $F\subseteq E(G^{(k-2)}[V_{-t}])$ with fewer than $h$ mutual extensions in $G$, the probability that this subset remains in $(G')^{(k-2)}[V_{-t}]$ is less than $\left(\frac{h}{n}\right)^{u} = \left(\frac{h}{n}\right)^{\lambda d}$. Additionally, the total number of such sets $F\subseteq E(G^{(k-2)}[V_{-t}])$ that are spanned by at most $d$ vertices is less than $\binom{kn}{d}\cdot 2^{\binom{d}{k-1}}$. Let $P$ be the event that $G'$ is $(h,d)$-vertex-extending to $V_t$, i.e., the event that $G'$ does not contain any such set $F$. Because the parameters satisfy the inequality $n^{-1} \geq \left(\frac{h}{n}\right)^{\lambda d}\binom{kn}{d}\cdot 2^{\binom{d}{k-1}}$,
\[\mb E[e(G') \mid P] \geq \mb E[e(G')] - \Pr[\overline{P}]\E[e(G') \mid \overline{P}] \geq p^{u+1}n^{k} - n^{k-1}.\]
Because $p<1/2$ and $d\ge 2$, we have that $p^{u+1}n^k\geq 2p^{u+2}n^k$ and $n^{k-1}\leq p^{(\lambda+1)d}n^{k}\leq p^{u+2}n^k$. Hence, $\E[e(G')\mid P]\geq p^{u+2}n^k$. It follows that there is a set $X_t$ of at most $u$ vertices such that the corresponding hypergraph $G'$ contains at least $p^{u+2}n^{k}\ge p^{(\lambda+1)d}n^k$ edges and is $(h, d)$-vertex-extending to $V_t$.

It remains to argue that $G'$ is also $(h, d)$-vertex-extending to $V_r$ for each $1\leq r < t$. Let $S \subseteq V_{-r}$ be a set of at most $d$ vertices. Then $S\cup X_t$ is a set of at most $d + u \leq (\lambda+1)d$ vertices in $V_{-r}$. Because $G$ is $(h,(\lambda+1)d)$-vertex-extending to $V_r$, there are at least $h$ vertices (in $V_r$) that are mutual extensions in $G$ of $E(G^{(k-2)}[S \cup X_t])$. For each mutual extension $v\in V_r$ and each edge $e\in E((G')^{(k-2)}[S])$, we claim that $e\cup \{v\}\in E(G')$. Indeed, because $e\in E((G')^{(k-2)}[V_{-r}])$, there exists some $b_r \in V_r$ such that $e\cup \{b_r\}\in E(G')$. Writing $\{e_t\}=e\cap V_t$, this implies that $(e\setminus \{e_t\}) \cup \{b_r, x_t\}\in E(G)$ for every $x_t\in X_t$. Hence, $(e\setminus \{e_t\}) \cup \{x_t\} \in E(G^{(k-2)}[S \cup X_t])$, implying that $(e\setminus \{e_t\}) \cup \{x_t,v\} \in E(G)$ for every $x_t \in X_t$. By construction, it follows that $e\cup\{v\}\in E(G')$. Thus, there are at least $h$ vertices (in $V_r$) that are mutual extensions in $G'$ of $E((G')^{(k-2)}[S])$, so $G'$ is $(h,d)$-vertex-extending to $V_r$.
\end{proof}

We state a corollary of the prior lemma that is easier to work with.

\begin{corollary}\label{lem:intermediate-steps}
Let $d\geq 4$, $k\geq 3$, $t$ be positive integers with $1\leq t \leq k$. Then for any $\epsilon\in (0, \frac{1}{3d}]$ and all $n \geq n_0=n_0(k, d, \epsilon)$, the following holds. Let $G$ be a $k$-partite $k$-uniform hypergraph on vertex set $V_1\sqcup\cdots\sqcup V_k$ with $n$ vertices in each part. Suppose that $G$ is $(n^{1/3}, 3d)$-vertex-extending to $V_r$ for all $r < t$ and that $e(G) \geq n^{k-\epsilon}$. 

Then there exists a subhypergraph $G'$ of $G$ on the same vertex set and a nonempty subset $X_t\subseteq V_t$ such that $G'$ is $(n^{1/3}, d)$-vertex-extending to $V_r$ for all $r \leq t$, we have $e(G')\geq n^{k-3d\epsilon}$, and, moreover, the edges of $G'$ are exactly those edges $e\in E(G)$ satisfying $(e\setminus V_t)\cup \{x_t\}\in E(G)$ for each $x_t\in X_t$.
\end{corollary}

\begin{proof}
We apply \cref{lem:intermediate-step-quantitative} with parameters $p = n^{-\epsilon}$, $\lambda = 2$, and $h = n^{1/3}$. For $n$ sufficiently large, the assumption $p\in (0, 1/2)$ is met. Moreover, $e(G) \ge n^{k-\epsilon} = pn^k$ and $G$ is $(n^{1/3}, 3d)$-vertex-extending to $V_r$ for all $r < t$.

It remains to check that the two inequalities are satisfied. Because $\epsilon \le \frac{1}{3d}$, it holds that $p^{(\lambda+1)d}n = n^{1-3d\epsilon}\ge 1$. Additionally,
\[n\left(\frac{h}{n}\right)^{\lambda d}\binom{kn}{d}2^{\binom{d}{k-1}} \le n\cdot n^{-4d/3}(kn)^d2^{\binom{d}{k-1}} = n^{1-d/3}\cdot k^d2^{\binom{d}{k-1}}.\]
Because $d\ge 4$, the right side is less than $1$ for all sufficiently large $n$. Thus, the hypotheses of \cref{lem:intermediate-step-quantitative} are satisfied.

Applying the lemma yields a subhypergraph $G'$ with $e(G') \ge p^{(\lambda+1)d}n^k = n^{k-3d\epsilon}$ satisfying the desired vertex-extending property.
\end{proof}

Applying \cref{lem:intermediate-steps} repeated $k$ times, one may show that there is a constant $c_k>0$ such that if $G$ is an $n$-vertex hypergraph with at least $n^{k-c_k/d^k}$ edges and $n$ is sufficiently large in terms of $k$ and $d$, then $G$ has a nonempty subhypergraph that is $(n^{1/3}, d)$-vertex-extending to every part. However, to reach $n^{k-c_k/d^{k-1}}$ (i.e., the bound in \cref{thm:k-uniform-pruning}), the last step must be handled slightly differently.

\begin{proof}[Proof of \cref{thm:k-uniform-pruning}]
We apply \cref{lem:intermediate-steps} repeated $k-1$ times to produce a sequence of hypergraphs $G = G_0 \supseteq G_1 \supseteq G_2\supseteq\cdots \supseteq G_{k-1}$. In detail, for each $t = 1, \ldots, k-1$, define $d_t = 3^{k-t+1}d$. Given sufficiently small $\epsilon_0>0$, define 
\[\epsilon_i = \left(\prod_{t=1}^{i} 3d_t\right)\epsilon_0.\]
For $\epsilon_0 \leq 3^{-k}3^{-\frac{k(k+1)}{2}}d^{-k+1} = 3^{-\frac{k(k+3)}{2}}d^{-k+1}$ we have $\epsilon_i\leq\tfrac1{3d}$ for all $1\leq i\leq k-2$ and $\epsilon_{k-1}\leq \tfrac13 < \tfrac12$.

Suppose $n$ is sufficiently large in terms of $k, d$, and that $e(G)\geq n^{k-\epsilon_0}$. Let $G_i$ be the subhypergraph of $G_{i-1}$ obtained by applying \cref{lem:intermediate-steps} with parameters $(d,k,t,\epsilon)=(d_i,k,i,\epsilon_i)$. If $e(G_{i-1})\geq n^{k-\epsilon_{i-1}}$, then $e(G_i)\geq n^{k-3d_i\epsilon_{i-1}}=n^{k-\epsilon_i}$.

After $k-1$ iterations, this process produces a hypergraph $G_{k-1}$ with $e(G_{k-1}) \geq n^{k-\epsilon_{k-1}} \geq n^{k-1/2}$ that is $(n^{1/3}, 9d)$-vertex-extending to $V_r$ for every $1\leq r\leq k-1$. Moreover, for each $1\leq r\leq k-1$, there exists a nonempty subset $X_r\subseteq V_r$ such that each edge $e\in E(G_{r-1})$ is in $E(G_r)$ if and only if $(e\setminus V_r)\cup \{x_r\}\in E(G_{r-1})$ for every $x_r\in X_r$. 

\begin{claim}\label{claim:edgesinproducts}
For every $e\in E(G_{k-1})$ we have $X_1\times \cdots \times X_{k-1}\times (e\cap V_{k})\subseteq E(G_{k-1})$. 
\end{claim}

\begin{proof}
For each $t$, let $\{e_t\} = e\cap V_t$. We first show that $\{e_1\}\times \cdots \times \{e_t\}\times X_{t+1}\times \cdots \times X_{k-1}\times \{e_k\} \subseteq E(G_t)$, using induction on $t = k-1, \dots, 0$. When $t = k-1$, this trivially holds. Suppose that this is true for $t\ge 1$, and we would like to show this for $t-1$.

Let $f\in \{e_1\}\times \cdots \times \{e_t\}\times X_{t+1}\times \cdots \times X_{k-1}\times \{e_k\}$. Because $f\in E(G_t)$, it holds that $(f\setminus V_t)\cup \{x_t\}\in E(G_{t-1})$ for all $x_t\in X_t$. This means that $\{e_1\}\times  \cdots\times \{e_{t-1}\} \times X_t\times X_{t+1}\times \cdots \times X_{k-1}\times \{e_k\} \subseteq E(G_{t-1})$, completing the induction.

We have established that $X_1\times \cdots \times X_{k-1}\times \{e_k\}\subseteq E(G_0)$. We now use induction to prove that $X_1\times \cdots \times X_{k-1}\times \{e_k\}\subseteq E(G_t)$ for all $0\leq t\leq k-1$. Suppose that this holds for $t-1$ where $t\geq 1$. We would like to show this for $t$. Let $f \in X_1\times \cdots \times X_{k-1}\times \{e_k\}$. By the inductive hypothesis, $(f \setminus V_t)\cup \{x_t\} \in X_1\times \cdots \times X_{k-1}\times \{e_k\} \subseteq E(G_{t-1})$ for all $x_t\in X_t$. Hence, by construction, we know that $f \in E(G_t)$. When $t=0$, this inductive argument establishes $X_1\times \cdots \times X_{k-1}\times \{e_k\}\subseteq E(G_{k-1})$, as desired.
\end{proof}
    
Let $B_k\subseteq V_k$ be the set of vertices incident to at least one edge of $G_{k-1}$. Then $n^{k-1}|B_k|\geq e(G_{k-1})\geq n^{k-1/2}$, so $|B_k| \geq n^{1/2}$. Also, by \cref{claim:edgesinproducts}, we have that $X_1\times \cdots \times X_{k-1}\times B_k\subseteq E(G_{k-1})$. Now choose $8d$ vertices uniformly at random from $B_k$ with replacement and let $X_k$ be the resulting set of vertices. Let $G_k$ be the subhypergraph of $G_{k-1}$ containing only those edges $e = \{e_1, \dots, e_k\}$ for which $(e\setminus V_k)\cup \{x_k\}$ is an edge of $G_{k-1}$ for each $x_k\in X_k$. (This is identical to the previous procedure except that $X_k$ is chosen from $B_k$ instead of from $V_k$.) We note that $X_1 \times X_2 \times \cdots \times X_k \subseteq E(G_k)$ by construction, so $G_k$ is not empty. To complete the proof we need to show that $G_k$ is $(n^{1/3},d)$-vertex-extending to each part. 

To show that $G_k$ is $(n^{1/3},d)$-vertex-extending to $V_k$ we rely on the argument used to prove \cref{lem:intermediate-step-quantitative}. There are fewer than $\binom{kn}{d}\cdot 2^{\binom{d}{k-1}}$ edge sets spanned by a size-$d$ vertex set in $V(G) \setminus V_k$. For each such set with at most $n^{1/3}$ mutual extensions in $G_{k-1}$, the probability that it is retained in $G_k$ is at most $(n^{1/3} / |B_k|)^{8d} \leq n^{-4d/3}$. Applying a union bound, the probability that $E(G_k)$ contains such a subset is at most $O(n^{-d/3})$. Hence, when $n$ is sufficiently large, $G_k$ is $(n^{1/3},d)$-vertex-extending to $V_k$ with positive probability.

It remains to prove that $G_k$ is $(n^{1/3},d)$-vertex extending to $V_r$ for each $1 \leq r < k$. Here too we mimic the proof of \cref{lem:intermediate-step-quantitative}. Fix $1 \leq r < k$ and consider any $S \subseteq V(G) \setminus V_r$ of size $d$. Set $S' = S \cup X_t$ and observe that $|S'| \leq 9d$. Recall that $G_{k-1}$ is $(n^{1/3},9d)$-vertex-extending to $V_r$.  Hence, there is a set $A_S \subseteq V_r$ of size at least $n^{1/3}$ such that, for every $e \in E(G_{k-1}^{(k-2)}[S'])$ and every $v \in A_S$, we have $e \cup \{v\} \in E(G_{k-1})$. We show that for every $e \in E(G_k^{(k-2)}[S])$ and every $v \in A_S$ it is the case that $e \cup \{v\} \in E(G_k)$, thus witnessing the vertex extension property of $G_k$. Let $e \in E(G_k^{(k-2)}[S])$. By definition, there exists a vertex $b \in V_r$ such that $e \cup \{b\} \in E(G_k)$. By construction, this implies that $(e \setminus \{e_k\}) \cup \{b,x\} \in E(G_{k-1})$, and hence $(e \setminus \{e_k\}) \cup \{x\} \in E(G_{k-1}^{(k-2)}[S'])$, for each $x \in X_k$. By the defining property of $A_S$, it holds that $(e \setminus \{e_k\}) \cup \{x,v\} \in E(G_{k-1})$ for each $v \in A_S$ and $x \in X_k$, so $e \cup \{v\} \in E(G_k)$ as desired.
\end{proof}

Finally we use the vertex-extending property guaranteed by \cref{thm:k-uniform-pruning} to embed any given hypergraph with bounded skeletal degeneracy. This requires a small trick since we only know that the $(k-1)$-edges extend to $k$-edges instead of the stronger property that all $i$-edges extend to $(i+1)$-edges.

\begin{theorem}
\label{thm:k-unif-turan-upper}
Let $H$ be a $k$-uniform $k$-partite hypergraph with $d_1(H)=d$. There exist constants $c_k>0$ (depending only on $k$) and $C_H$ such that every $n$-vertex $k$-uniform hypergraph on with at least $C_Hn^{k-c_k/d^{k-1}}$ edges contains a copy of $H$.
\end{theorem}

\begin{proof}
Because adding a constant number of vertices does not affect the statement up to the choice of $C_H$, we may assume that $n$ is a multiple of $k$.
We will apply \cref{thm:k-uniform-pruning} with parameters $k,k+d$. Let $c_k$ and $n_0$ be the constants produced by that result.

Let $H$ be a $k$-uniform $k$-partite hypergraph with $d_1(H) = d$. We may assume that $H$ is nonempty, in which case $d\geq k-1$. Let $C_H$ be a sufficiently large constant and set $c'_k = 3^{-k}c_k$. Let $G$ be a $k$-uniform hypergraph with $n$ vertices and at least $C_H n^{k-c'_k/d^{k-1}}$ edges. By setting $C_H$ large enough we may assume that the number of vertices of $H$ satisfies $n\geq k(v(H)+k)^3$ and $n\geq kn_0$. Indeed, the statement holds vacuously whenever $C_Hn^{k-c'_k/d^{k-1}}>n^k$.

We claim that $G$ contains a nonempty $k$-partite subhypergraph $G''$ with vertex partition $V(G)=V_1\sqcup\cdots\sqcup V_k$ that is $((n/k)^{1/3},d+k)$-vertex-extending to each part. Indeed, let $V_1\sqcup\cdots\sqcup V_k$ be an equitable partition of $V(G)$ (i.e., a partition in which $|V_i|=n/k$ for each $i$) chosen uniformly at random. Let $G' \subseteq G$ be the $k$-partite subhypergraph induced by this partition. Assuming that $C_H$ is sufficiently large, we have the inequality
\[\E[e(G')]=\frac{k!(n/k)^k}{n(n-1)\cdots(n-k+1)}\geq \frac{k!}{k^k}e(G) \geq n^{k-\frac{c'_k}{d^{k-1}}}\geq n^{k-\frac{c_k}{(d+k)^{k-1}}}.\]
Hence there is some choice of $G'$, a $k$-partite subhypergraph of $G$ with at least $n^{k- c_k/(d+k)^{k-1}}$ edges. We then apply \cref{thm:k-uniform-pruning} to $G'$ to produce $G''$.

Define $\hat H$ to be the $k$-uniform $k$-partite hypergraph formed from $H$ by adding a single auxiliary vertex to each part. Call the vertex partition $V(\hat H)=W_1\sqcup\cdots\sqcup W_k$ and let $\hat w_i\in W_i$ denote the auxiliary vertices. Let the edges of $\hat H$ be of the form
\[\{\hat w_i: i\in I\}\cup \left(e\cap \bigcup_{i\in[k]\setminus I} W_i\right)\]
for $e\in E(H)$ and $I\subseteq [k]$. In other words, the edges of $\hat H$ are formed from the edges of $H$ by replacing any number of vertices with the corresponding auxiliary vertices. We note that $d_1(\hat H)\leq d_1(H)+k$, because every edge present in $\hat H^{(1)}$ but not in $H^{(1)}$ contains at least one of the $k$ auxiliary vertices.

We now greedily construct an embedding $\phi\colon V(\hat H)\to V(G'')$ according to the degeneracy order. Let $v_1,\ldots,v_h$ be an ordering of $V(\hat H)$ in which $v_i=\hat w_i$ for $1\leq i \leq k$ and each subsequent vertex in the ordering is adjacent in $\hat H^{(1)}$ to at most $d+k$ previous vertices. We will define $\phi$ iteratively so that it maps vertices of $W_i$ to vertices of $V_i$ and edges of $\hat H$ to edges of $G''$.

First, define $\phi$ to map $\{v_1,\ldots,v_k\}$ to an edge of $G''$. We can do this because $G''$ is nonempty. Now suppose that we have embedded $v_1,\ldots,v_{j-1}$ for some $j>k$. Our goal is to embed $v_j\in W_i$ in a way that maintains these properties. Let $S\subseteq\{v_1,\ldots,v_{j-1}\}$ comprise those neighbors of $v_j$ in $\hat H^{(1)}$ preceding it in the ordering. Note that $|S|\leq d_1(\hat H)\leq d+k$. Because $G''$ is $((n/k)^{1/3},d+k)$-vertex extending, there are at least $(n/k)^{1/3}$ vertices $v\in V_i$ such that $e\cup\{v\}\in E(G'')$ for each $e\subseteq \phi(S)$ such that $e\in E(G''^{(k-2)}[V_{-i}])$. Pick $\phi(v_j)$ to be one of these vertices which has not already appeared in the image of $\phi$. We can do this because we assumed that $(n/k)^{1/3}\geq v(\hat H)=v(H)+k$.

We claim that $\phi$ still maps edges of $\hat H$ to edges of $G$ after this choice of $\phi(v_j)$. Let $e\cup\{v_j\}\in E(\hat H)$ be some edge. We know that $e\cup\{\hat w_i\}$ is a previously-embedded edge. In particular, we know that $\phi(e\cup\{\hat w_i\}) \in E(G'')$, so $\phi(e) \in E(G''^{(k-2)})$. By our choice of $\phi(v_j)$, we conclude that $\phi(e)\cup\{\phi(v_j)\} \in E(G'')$, as desired.

This embedding procedure terminates in a copy of $\hat H$ in $G''\subseteq G$. Because $H$ is a subhypergraph of $\hat H$, this completes the proof.
\end{proof}

\section{Lower bounds on the Tur\'an number}
\label{sec:turan-lower}

In this section we prove several lower bounds on the Tur\'an number. We give lower bounds for some specific classes of hypergraphs an well as a general lower bound in terms of the $i$-th skeletal degeneracy.

In this section we write $\mathbb{G}^k(n;p)$ for the Erd\H{o}s--R\'enyi random $k$-uniform hypergraph. In detail, $\mathbb{G}^k(n;p)$ has $n$ vertices and each of the $\binom nk$ possible $k$-edges is present independently with probability $p$.

We first give a lower bound for the Tur\'an number of $K^{(k)}_{d,d,\ldots,d}$. This result follows from standard probabilistic tools. We include the proof for completeness, since this lower bound on the Tur\'an number of $K^{(k)}_{d,d,\ldots,d}$ shows that the upper bound in \cref{thm:k-unif-turan-main} is tight (up to the constant $c_k$) in some cases.

\begin{proposition}
\label{prop:complete-k-unif-k-part-lower}
For each $d,k\geq 2$ we have
\[\ex(n,K^{(k)}_{d,d,\ldots,d})\geq\Omega_k(n^{k-k/d^{k-1}})\]
while $d_1(K^{(k)}_{d,d,\ldots,d})=(k-1)d$.
\end{proposition}

\begin{proof}
For ease of notation write $H$ for $K^{(k)}_{d,d,\ldots,d}$. Then the $1$-skeleton $H^{(1)}$ is the complete $k$-partite graph with $d$ vertices in each part. Since $H$ is $(k-1)d$-regular, we see that $\skel{H} = (k-1)d$.

Let $p = n^{-k/d^{k-1}}$ and let $G'\sim \mb G^{k}(n; p)$. Then we have $\mb E[e(G')] = \binom{n}{k}p$. On the other hand, the number $X$ of labeled copies of $H$ in $G'$ has expectation $\mb E[X] \leq n^{kd}p^{d^{k}} = 1$. Therefore there exists $G'$ such that \[e(G') - X \geq \binom{n}{k}p - 1 \geq \frac{1}{k^k}n^{k-\frac{k}{d^{k-1}}}-1.\] We may remove an edge for each copy of $H$ in $G'$ to get $G$. Now $G$ contains no copy of $H$ and has at least $\Omega_k(n^{k-k/d^{k-1}})$ edges.
\end{proof}

Next we prove a general lower bound for the Tur\'an number in terms of the $i$-th skeletal degeneracy for each $i$. The $i=1$ case of the following theorem implies the lower bound in \cref{thm:k-unif-turan-main}.

\begin{theorem}
\label{thm:k-uniform-construction-ith-degen}
Let $d, k, i$ be positive integers where $1 \leq i < k$ and $d> \binom{k}{i+1}$. Let $H$ be any $k$-uniform hypergraph with $d_i(H) \geq d$. Then \[\ex(n,H)\geq \Omega_k( n^{k- 3^k / d}).\]
\end{theorem}

\begin{proof}
Suppose that $n$ is sufficiently large in terms of $k$. Write $r=i+1$ and set $p = n^{-r/d}$. Let $G_0\sim \mathbb{G}^r(n, p)$ and let $X$ be the number of copies of $K_k^{(r)}$ in $G_0$. We have $\mb E[X] = \binom{n}{k}p^{\binom{k}{r}}$. Since $d_i(H) \geq d$, there is an $r$-uniform hypergraph $F \subseteq H^{(i)}$ on $v$ vertices with minimum degree at least $d$. This implies that $e(F) \geq \frac{vd}{r}$. The number $Y$ of copies of $F$ in $G_0$ is in expectation at most $\mb E[Y]\le n^vp^{e(F)} \leq 1$. Hence there exists a choice of $G_0$ such that 
\[Z = X-\binom{n}{k-r}Y \ge \mb E\left[X-\binom{n}{k-r}Y\right]\geq \left(\frac nk\right)^k n^{-\frac rd \binom kr}-n^{k-r} \geq \Omega_k \left(n^{k- r\binom{k}{r} / d} \right).\]
Here, we have used the assumption that $d > \binom{k}{r}$. We note that $r\binom{k}{r} / d \leq 3^k / d$, so $Z \geq \Omega_k(n^{k - 3^k/d})$.

We construct $G_1$ from $G_0$ by removing a single $r$-edge for each copy of $F$ in $G_0$. By doing this the number of $r$-edges we remove is at most $Y$, and therefore the number of copies of $K_k^{(r)}$ we remove is at most $\binom n{k-r}Y$. Hence $G_1$ has at least $\Omega_k \left(n^{k- 3^k / d} \right)$ copies of $K_k^{(r)}$. Finally, we construct the $k$-uniform hypergraph $G$ by putting a $k$-edge on every copy of $K_k^{(r)}$ in $G_1$. By construction, $G^{(i)} \subseteq G_1$ does not contain any copies of $F$, and therefore contains no copies of $H$, as desired.
\end{proof}

Note that examining the bound above more carefully in the $i=1$ case gives the following result with a better constant in the exponent.

\begin{theorem}
\label{thm:k-uniform-construction}
For $k\geq 2$ and $d > k(k-1)/2$, every $k$-uniform hypergraph $H$ satisfying $\skel{H} \geq d$ has
\[\ex(n,H)\geq\Omega_k(n^{k-k(k-1)/d}).\]
\end{theorem}

We now give the details of the construction promised in the introduction of a family of $k$-uniform $k$-partite hypergraphs that are all 1-degenerate (but not 1-skeletal degenerate) and with Tur\'an exponent $k-o(1)$. By \cref{thm:k-uniform-construction} (or \cref{thm:k-uniform-construction-ith-degen}) it suffices to exhibit a family of $k$-uniform $k$-partite hypergraphs $H$ with $d_{k-1}(H)=1$ but $d_1(H)$ arbitrarily large.

\begin{definition}
\label{defn:bipartite-hedgehog}
For $k\geq 2$ and $d\geq 1$ define the \emph{bipartite hedgehog} $H_d^{(k)}$ to be the $k$-uniform $k$-partite hypergraph on $2d+(k-2)d^2$ vertices formed by extending each edge of $K_{d,d}$ to a $k$-edge.\footnote{The typical hedgehog hypergraphs are formed by applying the same procedure to $K_d$ instead of $K_{d,d}$.}
Formally, $H_d^{(k)}$ has vertices $[d]\sqcup[d] \sqcup \bigsqcup_{i=1}^{k-2} [d]^2$ and the $d^2$ $k$-edges $(i,j,(i,j),\ldots,(i,j))$ for each $(i,j)\in[d]^2$. 
\end{definition}

\begin{corollary}
    \label{cor:degen-turan-counterex}
For each $k\geq 3$ and $d>k(k-1)/2$ we have
\[\ex(n,H_d^{(k)})\geq \Omega_{k}(n^{k-k(k-1)/d})\] while $d_{k-1}(H_d^{(k)})=1$. 
\end{corollary}

\begin{proof}
Clearly $H_d^{(k)}$ is 1-degenerate; any subhypergraph of $H_d^{(k)}$ either contains one of the vertices in the latter $k-2$ parts, which all have degree at most 1, or contains none of these vertices and thus is the empty hypergraph.

However one can also easily see that $d_1(H_d^{(k)})\geq d$, since the 1-skeleton of $H_d^{(k)}$ contains $K_{d,d}$. Therefore by \cref{thm:k-uniform-construction} we have the desired bound on the Tur\'an number of $H_d^{(k)}$.
\end{proof}

For completeness, we also give a matching upper bound on the Tur\'an exponent of $H_d^{(k)}$. The following result shows that the lower bound in \cref{thm:k-unif-turan-main} is sometimes of the correct shape.

\begin{proposition}
\label{prop:bipartite-hedgehog-upper}
For every $d\geq k\geq 2$ we have
\[\ex(n,H_d^{(k)})\leq O_{d,k}(n^{k-1/d})\]
and $d_1(H_d^{(k)})=d$.
\end{proposition}

\begin{proof}
Let $n$ be sufficiently large in terms of $d,k$. Let $G$ be a $k$-uniform hypergraph with $n$ vertices and at least $Cn^{k-1/d}$ edges for some large constant $C$ depending on $d,k$. Define $G'$ to be the graph on the same vertex set whose edges are the pairs of vertices that are contained in more than $kd^2n^{k-3}$ many $k$-edges in $G$. The number of edges in $G'$ is at least
\[\frac{Cn^{k-1/d}\binom k2 -kd^2n^{k-3}\binom n2}{\binom {n-2}{k-2}}>2n^{2-1/d}\]
where the last inequality holds for an appropriate choice of $C$ and all sufficiently large $n$. By the K\H{o}v\'ari--S\'os--Tur\'an theorem there is a copy of $K_{d,d}$ in $G'$.

We now claim that this copy of $K_{d,d}$ in $G'$ extends to a copy of $H_d^{(k)}$ in $G$. Recall that $H_d^{(k)}$ has vertex set $V=V_1\sqcup\cdots\sqcup V_k$ where $V_1=V_2=[d]$ and $V_3=\cdots=V_k=[d]^2$ and edge set $E=\{e_{i,j}\}_{i,j\in[d]}$ where $e_{i,j}=(i,j,(i,j),\ldots,(i,j))$. Pick $\phi\colon V_1\sqcup V_2\to V(G)=V(G')$ to be a copy of $K_{d,d}$ in $G'$.

Now we extend $\phi$ to a map $V(H_d^{(k)})  \to V(G)$ one edge at a time. Suppose that at some stage of this process we have a partial map $\phi$ and we want to extend $\phi$ to $e_{i,j}=(v_1,v_2,\ldots,v_k)$. The number of vertices appearing so far in the image of $\phi$ is at most $v(H_d^{(k)})<kd^2$. Since $\phi(v_1)\phi(v_2)$ is an edge of $G'$ there are more than $kd^2n^{k-3}$ edges of $G$ containing $\phi(v_1)\phi(v_2)$. We wish to avoid edges that contain vertices in previously embedded vertices from $V_3\sqcup\cdots\sqcup V_k$. Since there are at most $kd^2$ of these vertices, there are at most $kd^2\binom n{k-3}$ edges we are not allowed to use. Therefore there is at least one valid edge. We extend the definition of $\phi$ to $v_3,\ldots,v_k$ so that $\{\phi(v_1),\phi(v_2),\phi(v_3),\ldots,\phi(v_k)\}$ is this edge.

Since this process cannot get stuck, it produces a copy of $H_d^{(k)}$ in $G$, as desired.

We already saw that $d_1(H^{(k)}_d)\geq d$. This is an equality for $k\geq d$, since any subgraph of the 1-skeleton of $H^{(k)}_d$ either contains one of the vertices in the latter $k-2$ parts, which all have degree at most $k-1$, or contains none of these vertices and then is a subgraph of $K_{d,d}$. Thus $d\leq d_1(H^{(k)}_d)\leq\max\{k-1,d\}$.
\end{proof}

\section{Almost-linear Ramsey number}
\label{sec:ramsey-almost-lin}

In this section we use the techniques of \cref{sec:turan-upper} to give a short argument that proves that hypergraphs with bounded skeletal degeneracy have almost-linear Ramsey number. We will improve this bound to linear via a more complicated argument in the next section.

\begin{theorem}
\label{thm:almost-lin-ramsey}
For $k,d,q\geq 2$, let $H$ be a $k$-uniform hypergraph on $n$ vertices with $\skel{H}=d$. Then the $q$-color Ramsey number of $H$ satisfies $r(H;q)\leq n^{1+o_{k,d,q;n\to\infty}(1)}$.
\end{theorem}

To prove \cref{thm:almost-lin-ramsey}, we obtain a different corollary of \cref{lem:intermediate-step-quantitative} by choosing different values for the parameters.

\begin{corollary}
\label{cor:almost-lin-turan}
For $p\in (0, 1/2)$ and $k, d\geq 1$, there exists a constant $c_{p,k,d}>0$ such that the following holds. Suppose that $n\geq n_0=n_0(p,k,d)$ and let $G$ be a $k$-partite $k$-uniform hypergraph with $n$ vertices in each part and at least $pn^k$ edges. Then $G$ contains every $k$-partite $k$-uniform hypergraph $H$ with $\skel{H}=d$ on $h$ vertices where
\[h \le n^{1-c_{p, k, d} (\log n)^{-\frac{2}{k(k+1)}}}.\]
\end{corollary}

\begin{proof}
We first apply \cref{lem:intermediate-step-quantitative} repeated $k$ times starting from $G$ producing a sequence of hypergraphs $G=G_0\supseteq G_1\supseteq\cdots\supseteq G_k$. In detail, define $d'=d+k$, let $\lambda\geq 1$ be specified later, and set $d_t=(\lambda+1)^{k-t}d'$. We apply \cref{lem:intermediate-step-quantitative} to $G_{t-1}$ with parameters $(d_t,k,t,h,n,\lambda)$ to obtain $G_t$. 

We may assume that $n$ is sufficiently large with respect to $k, d$. This implies that the second inequality in the hypothesis of \cref{lem:intermediate-step-quantitative} holds as long as $h \leq n^{1-\frac{2}{\lambda}}$.

Define \[p_t=p^{(\lambda+1)^{\binom{t+1}2}(d')^t}\] for $0\leq t\leq k$. By hypothesis $G_0$ has at least $p_0n^k$ edges. We see that if $e(G_{t-
1})\geq p_{t-1}n^k$ and $p_{t-1}^{(\lambda+1)d_{t-1}}n\geq 1$, then the first inequality in the hypothesis of \cref{lem:intermediate-step-quantitative} is satisfied, so then $e(G_{t})\geq p_{t-1}^{(\lambda+1)d_{t-1}}n^k= p_{t}n^k$.

All of these inequalities are satisfied if \[p^{(\lambda+1)^{\binom{k+1}{2}}(d')^k}n\geq 1.\] Therefore we can take any \[\lambda+1 \ge \left(\frac{\log n}{(d')^k\log 1/p}\right)^{\frac{2}{k(k+1)}} = c'_{p, d, k} (\log n)^{\frac{2}{k(k+1)}}.\] Now we know that $G_k$ is a $k$-uniform $k$-partite hypergraph that is $(h, d+k)$-vertex-extending to each part and contains at least one edge. Finally, we find an embedding of $H$ in $G_k$ as in the proof of \cref{thm:k-unif-turan-upper}.
\end{proof}

Using an argument of Kostochka and R\"{o}dl \cite{KR06} we can deduce an almost linear bound on the Ramsey number from this result. See also \cite{CFS09} for another instance of this argument.

\begin{lemma}[{cf.\ \cite[Proof of Theorem 4]{KR06}}]
\label{lem:kr-reduction}
Let $H$ be a $k$-uniform hypergraph on $n$ vertices with skeletal degeneracy $d_1(H)=d$ whose 1-skeleton is $\ell$-colorable. For a positive integer $q$ and a $q$-coloring $f\colon E(K_N^{(k)}) \to [q]$ of the edges of $K_N^{(k)}$, there exist
\begin{itemize}
    \item a constant $p_{k,\ell,d} > 0$ depending only on $k,\ell,d$;

    \item an $\ell$-uniform $\ell$-partite hypergraph $\hat H$ with at most $(1+(\ell-k)d^{k-1})n$ vertices and $d_1(\hat H)=d$; and

    \item an $\ell$-uniform $\ell$-partite hypergraph $\hat G$ with parts of size $N' = \lfloor N/\ell\rfloor$ and at least $p_{k,\ell,q}N'^\ell$ edges
\end{itemize}
such that if $\hat G$ contains a copy of $\hat H$ then $f$ contains a monochromatic copy of $H$.
\end{lemma}

The proof uses the following easy upper bound on the number of edges in a degenerate hypergraph.

\begin{lemma}
\label{lem:degen-edge-upper}
Let $H$ be a $k$-uniform $n$-vertex hypergraph. Then
\[ e(H) \leq d_1(H)^{k-1}n. \]
\end{lemma}

\begin{proof}
A well-known characterization of degeneracy of graphs implies that if $G$ is a $d$-degenerate graph, then there is an ordering $v_1,v_2,\ldots,v_n$ of $V(G)$ such that each vertex $v_i$ is adjacent to at most $d$ of the vertices $v_1,\ldots,v_{i-1}$.

We apply this characterization to the 1-skeleton $H^{(1)}$. We count the number of $k$-edges of $H$ whose last vertex in this degeneracy ordering is $v_j$. Then the other $k-1$ vertices of the edge must lie in $v_1,\ldots,v_{j-1}$ and be adjacent to $v_i$ in $H^{(1)}$. Thus the number of $k$-edges of $H$ whose last vertex is $v_j$ is at most $\binom {d_1(H)}{k-1}$. Summing over all $n$ choices of $j$ gives the claimed bound.
\end{proof}

\begin{proof}[Proof of \cref{lem:kr-reduction}]
Note that the result is vacuously true if $N<\ell$. Thus we assume that $N\geq\ell$. We define an $\ell$-uniform $\ell$-partite hypergraph $\hat H$ as follows. As the $1$-skeleton of $H$ is $\ell$-colorable, there is a partition of $V(H)$ into $\ell$ parts such that every edge of $H$ has vertices from $k$ different parts of the partition. Then for each edge of $H$, add an $\ell$-edge to $\hat H$ consisting of the $k$ vertices of the original edge and $\ell-k$ auxiliary vertices. Place the auxiliary vertices in the parts such that the edge uses one vertex from each part. At the end of the process we have added $e(H)(\ell-k)$ auxiliary vertices in total, each used in exactly one edge. This gives an $\ell$-uniform $\ell$-partite hypergraph, as desired.

Note that $\skel{\hat H}=\skel{H}$ using the same degeneracy ordering, and then placing all the auxiliary vertices at the end of the ordering. Furthermore since \cref{lem:degen-edge-upper} gives $e(H)\leq d^{k-1}n$, we have the desired bound on $v(\hat H)$.

Let $r=r(K_\ell^{(k)};q)$ be the $q$-color Ramsey number of the complete $k$-uniform hypergraph on $\ell$ vertices. Given the $q$-coloring $f$ of $K_N^{(k)}$ there are at least
\[\frac{\binom Nr}{\binom {N-\ell}{r-\ell}}\geq \frac{N^\ell}{r^\ell}\]
monochromatic copies of $K_\ell^{(k)}$ in the coloring, at least a $1/q$-fraction of which share a color $i \in [q]$. Define $G^+$ to be the $\ell$-uniform hypergraph on $N$ vertices whose $\ell$-edges correspond to the $i$-colored cliques in $f$.

Finally, consider a uniformly random partition of $V(G^+)$ into $\ell+1$ parts $V_1,\ldots,V_\ell,V_{\ell+1}$, where $|V_1|=|V_2|=\ldots=|V_\ell|=\lfloor N/\ell \rfloor$. Let $\tilde G$ be the (random) hypergraph with vertex set $\bigcup_{i=1}^\ell V_i$ and all edges of $G^+$ that have one vertex in each of $V_1,\ldots,V_\ell$. Since $N\geq\ell$, each edge of $G^+$ is in $\tilde G$ with probability at least $\ell! (2\ell)^{-\ell} \geq (2e)^{-\ell}$. Thus, there exists an $\ell$-uniform $\ell$-partite hypergraph $\hat G \subseteq G^+$ with $\lfloor N/\ell \rfloor$ vertices in each part and that contains at least a $(2e)^{-\ell}$-fraction of the edges in $G^+$. As $G^+$ has at least $\frac{1}{q}(N/r)^\ell$ edges, it follows that $\hat G$ has at least $pN^\ell$ edges with $p=1/(q(2er)^{\ell})$.
\end{proof}

\begin{remark}
In \cref{lem:kr-reduction}, we can always take $\ell\leq d+1$ since the chromatic number of a graph is always at most one more than its degeneracy. For fixed $k\geq 3$ and $q$, the value of $p_{k,\ell,q}^{-1}$ grows like a tower of height $k$ with $C_{k,q}d$ at the top. This growth mainly comes from the dependence on the Ramsey number $r(K_\ell^{(k)};q)$. 
\end{remark}

Now \cref{thm:almost-lin-ramsey} follows by combining \cref{cor:almost-lin-turan,lem:kr-reduction}.

\begin{proof}[Proof of \cref{thm:almost-lin-ramsey}]
Let $N$ be a positive integer and let $f\colon K_N^{(k)}\to[q]$ be a $q$-coloring of the edges of $K_N^{(k)}$. We apply \cref{lem:kr-reduction} to the hypergraph $H$ with $\ell = d+1$. This produces an $\ell$-uniform $\ell$-partite hypergraph $\hat{H}$ with at most $O_{\ell, k, d}(n)$ vertices and $d_1(\hat{H}) = d$ and an $\ell$-uniform $\ell$-partite hypergraph $\hat G$ with parts of size at least $N' = \lfloor N/\ell\rfloor$ and at least $p_{k,\ell, q}(N')^{\ell}$ edges. These hypergraphs have the property that there is a monochromatic copy of $H$ in the $q$-coloring $f$ if there is a copy of $\hat{H}$ in $\hat G$.

We apply \cref{cor:almost-lin-turan} to $\hat G$ and $\hat H$ with $(n,k,d, p)=(N', \ell, d, p_{k,\ell, q})$. This shows that $\hat G$ contains a copy of $\hat H$ if
\[v(\hat H) \le (N')^{1-c(\log N')^{-\frac{2}{\ell(\ell+1)}}}\] 
holds where $c$ is the constant depending on $p_{k,\ell, q}, \ell, d$ in \cref{cor:almost-lin-turan}. Thus in our setting $c$ only depends on $k, d, q$. Now using that $v(\hat H) \leq O_{\ell, k, d}(n) = O_{k,d}(n)$, and $N' = \lfloor N/\ell\rfloor =\lfloor N/(d+1)\rfloor$, we conclude that the above inequality holds as long as
\[N \ge Cn^{1+c'(\log n)^{-\frac{2}{(d+1)(d+2)}}}\]
holds for some appropriately chosen $c', C$ depending only on $k,d,q$. As long as this inequality is satisfied, for every $q$-coloring $f$ of the edges of $K_N^{(k)}$, we know that $\hat G$ always contain a copy of $\hat H$, so $f$ always contains a monochromatic copy of $H$.
\end{proof}

\section{Linear Ramsey number}
\label{sec:ramsey-lin}

In this section we prove that every $k$-uniform hypergraph on $n$ vertices with bounded skeletal degeneracy has Ramsey number linear in $n$. This involves a refinement of the previous Tur\'an result that uses a notion of defect that carefully keeps track of vertex sets that do not extend sufficiently.

Recall from the introduction that Lee \cite{Lee17} proved the graphical (i.e., $k=2$) case of \cref{thm:lin-ramsey-main}. Our main innovation is introducing a new notion of defect for hypergraphs. We use dependent random choice to prune a given hypergraph while controlling this notion of defect. The hypergraph $H$ can then be embedded into the pruned hypergraph. The embedding step is essentially identical to the one used by Lee \cite{Lee17} in the graphical case.

Similarly to the previous section our main Ramsey result, \cref{thm:lin-ramsey-main}, is deduced from the following Tur\'an result via \cref{lem:kr-reduction}.

\begin{theorem}
\label{thm:lin-turan}
Let $p \in (0, 1/2)$ and let $k, d$ be positive integers. There exists a constant $c_{p,k,d}>0$ such that the following holds. Let $G$ be a $k$-partite $k$-uniform hypergraph with $n$ vertices in each part and at least $pn^k$ edges. Then $G$ contains every $k$-partite $k$-uniform $H$ with $\skel{H}=d$ on $h$ vertices where $h \le c_{p, k, d}n$.
\end{theorem}

\begin{proof}[Proof of \cref{thm:lin-ramsey-main} assuming \cref{thm:lin-turan}]
Let $N$ be a positive integer and let $f\colon E(K_N^{(k)}) \to [q]$ be a $q$-coloring of the edges of $K_N^{(k)}$. Since, by assumption, $H^{(1)}$ is $d$-degenerate it is $\ell = (d+1)$-colorable. Thus, by \cref{lem:kr-reduction} there exists an $\ell$-uniform $\ell$-partite hypergraph $\hat{H}$ with at most $(1+(\ell-k)d^{k-1})n$ vertices and $d_1(\hat{H}) = d$, and an $\ell$-uniform $\ell$-partite hypergraph $\hat G$ with parts of size $N' = \lfloor N/\ell\rfloor$ and at least $p_{k,\ell, q}(N')^{\ell}$ edges, such that if there is a copy of $\hat{H}$ in $\hat G$ then $f$ contains a monochromatic copy of $H$.

Applying \cref{thm:lin-turan} with $(n,k,d,p) = (N', \ell, d, p_{k,\ell, q})$ we conclude that $\hat G$ contains a copy of $\hat H$ if $v(\hat H) \le cN'$, where $c$ is the constant depending on $p_{k,\ell, q}, \ell, d$ from \cref{thm:lin-turan}. Note that $c$ only depends on $k, d, q$. Using the facts that $v(\hat H) \le (1+(\ell-k)d^{k-1})n$ and $N' = \lfloor N/\ell \rfloor = \lfloor N/(d+1) \rfloor$, we conclude that $v(\hat H)\le cN'$ as long as $n\le c'N$ for $c' = \frac{c}{2(d+1)(1+(\ell-k)d^{k-1})}$, a constant depending only on $k,d,q$. As long as this inequality is satisfied, $\hat G$ contains a copy of $\hat H$ which in turn implies that $f$ contains a monochromatic copy of $H$. Hence the desired result holds with constant $1/c'$, which only depends on $k,d,q$.
\end{proof}

We spend the remainder of this section proving the Tur\'an result. In \cref{sec:pruning} we define our notion of defect for hypergraphs and prove that we can prune $G$ in a way that gives us control over the defect. In \cref{sec:embedding} we show that we can embed $H$ into $G$. In \cref{sec:proof-of-main} we combine these two results to prove \cref{thm:lin-turan}. Almost all of the novelty lies in \cref{sec:pruning} while the arguments in \cref{sec:embedding,sec:proof-of-main} closely follow the strategy used in \cite{Lee17}.

\subsection{Pruning}
\label{sec:pruning}

Let $G$ be a $k$-uniform $k$-partite hypergraph with vertex parts $V_1,\ldots,V_k$. Define
\[E^*(G)=\{e\subseteq V(G): e\subseteq f\text{ for some }f\in E(G)\}.\] 
We refer to $E^*(G)$ as the \emph{partial edges} of $G$. Note that $\emptyset\in E^*(G)$ for any nonempty hypergraph $G$. 

For $Q\subseteq V(G)$, write 
\[PE_Q(G)= \{s\subseteq Q: s\in E^*(G)\}\]for the partial edges of $G$ contained in $Q$. For $E\subseteq E^*(G)$ define the \emph{common neighborhood} of $E$ to be
\[N(E; G)=\{v\in V(G) : e\sqcup\{v\}\in E^*(G) \text{ for all }e\in E\}.\]
For $Q\subseteq V(G)$ we will often consider $N(PE_Q(G);G)$, the common neighborhood of $Q$. In words this is the set of vertices $v$ such that every partial edge in $Q$ extends to another partial edge with the addition of the vertex $v$.

Finally for $\theta > 0$ define the \emph{defect}
\begin{equation}\label{eq:omega_theta set def}
\omega_\theta(Q, V_i; G) = \omega_\theta(|N(PE_Q(G); G)\cap V_i|),
\end{equation}
where $\omega_\theta\colon\mathbb R_{\geq 0}\to\mathbb R_{\geq0}\cup\{+\infty\}$ is defined by
\begin{equation}\label{eq:omega_theta function def}
\omega_{\theta}(x)=\begin{cases}
+\infty \quad &\text{if }x = 0,\\
\frac{\theta}{x}\quad&\text{if }0<x<\theta,\\
0&\text{otherwise}.
\end{cases}
\end{equation}

The defect $\omega_\theta(Q,V_i;G)$ measures if there are at least $\theta$ vertices of $V_i$ in the common neighborhood of $Q$. The defect is 0 when the common neighborhood contains at least $\theta$ vertices of $V_i$, the defect is infinite when the common neighborhood has zero vertices of $V_i$, and for intermediate sizes it interpolates between these values.

Given a partition $V_1\sqcup\cdots\sqcup V_k$, recall that we write $V_{-i}=\bigcup_{j\neq i}V_j$.
We will often consider $\omega_\theta(Q,V_i;G)$ for $Q\in (V_{-i})^d$. (Note that we abuse notation and view a tuple as its set of coordinates.)

Note that if $Q\subseteq Q'$, then $PE_Q(G) \subseteq PE_{Q'}(G)$, so $N(PE_Q(G);G)\cap V_i\supseteq N(PE_{Q'}(G);G)\cap V_i$, and thus $\omega_\theta(Q, V_i; G)\le \omega_\theta(Q', V_i; G)$.

We now state our main dependent random choice lemma.

\begin{lemma}
\label{lem:simul-pruning}
Suppose that $G$ is a $k$-partite $k$-uniform hypergraph on $V_1\sqcup \cdots \sqcup V_k$ with at least $p|V_1| \cdots |V_k|$ edges. Let $d, t$ be positive integers and $\theta$ be a positive real number. For each $1\leq i\leq k$, let $\bm{X}_i = ((\bm{X}_i)_{u})_{u\in [t]}\in V_i^t$ be a uniform random $t$-tuple of vertices. Let $G'\subseteq G$ be chosen in the following way: for every edge $e = (e_1, \dots, e_k)\in E(G)$, we keep $e$ in $G'$ if and only if all edges in $(\{e_1\}\cup \bm{X}_1)\times \cdots \times (\{e_k\}\cup \bm{X}_k)$ are in $E(G)$. Next, remove all isolated vertices to get the random $k$-partite $k$-uniform hypergraph $G'$ with parts $V_1'\sqcup \cdots \sqcup V_k'$. Then $\mathbb{E}[e(G')]\geq p^{(t+1)^k}|V_1|\cdots |V_k|$, and for each $i\in [k]$,
\[\mb{E}\left[\sum_{Q\in (V_{-i}')^d}\omega_\theta(Q, V_i'; G')^t\right] \leq 2^{2^d} \left(\frac{\theta}{|V_i|}\right)^t v(G)^d.\]
\end{lemma}

\begin{remark}
To prove the last inequality in the lemma, we will show that for each $d$-tuple of vertices $Q$, we have $\mb{E}\left[\omega_\theta(Q, V_i'; G')^t\right] \leq 2^{2^d}\left(\frac{\theta}{|V_i|}\right)^t$. In other words, we can control up through the $t$-th moment of the defect for each $Q$. Intuitively (and imprecisely), the reason for this is that if some $Q\in(V_{-i})^t$ has large defect, say $\theta/m$, then $Q$ has a small (size $m$) common neighborhood. In this case it is unlikely for $Q$ to survive into $G'$, since all $t$ of the vertices $(\bm X_i)_1,\ldots,(\bm X_i)_t$ would have to be chosen from the common neighborhood. Hence, although $Q$ would have $t$-th moment of defect $(\theta/m)^t$, the probability that $Q$ survives into $G'$ is only $(m/|V_i|)^t$. This gives the $(\theta/|V_i|)^t$ contribution in the inequality.

In the next sections we will take $\theta / v(G)$ to be a sufficiently small constant in terms of $p$. The inequality then implies that on average over $Q \in (V_{-i}')^d$ we expect $\omega_\theta(Q,V_i';G)^t$ to be small. Heuristically, this implies that for a large fraction of $d$-sets in $V'_{-i}$ the number of common neighbors is linear in $\theta$, which is itself linear in $v(G)$. This property will be crucial when we embed hypergraphs into $G$ vertex by vertex.
\end{remark}

\begin{proof}[Proof of \cref{lem:simul-pruning}]
Note that if any edge in $\bm{X}_1\times \cdots \times \bm{X}_k$ is not present in $G$ then $G'$ is empty. In what follows, expectations are taken over $(\bm{X}_1, \ldots, \bm{X}_k)$ unless specified otherwise.

First we compute the expectation of $e(G')$. Note that we can see this as a $k$-step procedure of deleting edges: $G = G_0 \leadsto G_1 \leadsto \cdots \leadsto G_k= G''\leadsto G'$, where in step $i\geq 1$ the hypergraph $G_i$ is derived from $G_{i-1}$ by keeping only edges $e\in G_{i-1}$ if $(e\cap V_{-i})\cup (\bm X_i)_u \in E(G_{i-1})$ for all $1\leq u\leq t$. Finally $G'$ is derived from $G''$ by removing isolated vertices, so $E(G') = E(G'')$. 

Condition on $G_{i-1}$. For each $e\in \prod_{j\ne i}V_j$, note that all $|N(\{e\}; G_{i-1})|$ edges containing $e$ are kept in $G_{i}$ if and only if $\bm X_i$ are all in $N(\{e\}; G_{i-1})$. First, we have
\[\sum_{e\in \prod_{j\ne i}V_j} |N(\{e\}; G_{i-1})| = e(G_{i-1}).\]
Second, since the probability that $\bm X_i$ are all in $N(\{e\}; G_{i-1})$ is $\left(\frac{|N(\{e\}; G_{i-1})|}{|V_i|}\right)^t$, we obtain
\[ \mb E_{\bm X_i}[e(G_i)] = \sum_{e\in \prod_{j\ne i}V_j}|N(\{e\}; G_{i-1})|\left(\frac{|N(\{e\}; G_{i-1})|}{|V_i|}\right)^t \geq \frac{e(G_{i-1})^{t+1}}{\prod_{j=1}^k|V_j|^t}.\]

As a consequence, combining this for all $i = k, k-1, \ldots, 1$, we have
\[\mb E\left[\frac{e(G'')}{\prod_{j=1}^k|V_j|}\right] \geq \mb E\left[\left(\frac{e(G_{k-1})}{\prod_{j=1}^k|V_j|}\right)^{(t+1)}\right] \geq \cdots \geq \left(\frac{e(G)}{\prod_{j=1}^k|V_j|}\right)^{(t+1)^k},\]
as desired.

Now we aim to establish the second inequality. By symmetry we may assume that $i = 1$. For each $Q\in V_{-1}^d$ and each nonempty subset $E\subseteq PE_Q(G)$, we first aim to bound
\[\E\left[\omega_\theta(|N(E; G')\cap V_1'|)^t\mathbf{1}\{E \subseteq PE_Q(G')\}\right].\]
For each $s\in E$, the event $s\in PE_{Q}(G')$ occurs if and only if every $k$-tuple in $\prod_{\ell=1}^k \bm X_\ell\cup(s\cap V_\ell)$ lies in $E(G)$. Note that since $Q\in(V_{-1})^d$, we know that $s\cap V_1=\emptyset$. Define $\bm{F}_E = \bigcup_{s\in E}\prod_{\ell=2}^k \bm{X}_\ell\cup (s\cap V_\ell)$. Thus the event $E\subseteq PE_Q(G')$ occurs if and only if $\bm X_1\times \bm F_E\subseteq E(G)$ which occurs if and only if $\bm X_1\subseteq N(\bm F_E;G)$. Furthermore, if $\bm X_1\subseteq N(\bm F_E;G)$, then $N(E;G')\cap V_1'\supseteq N(\bm F_E;G)$. Therefore, conditioning on $\bm F_E$, we compute
\begin{align*}
\E_{\bm X_1}\left[\omega_\theta(|N(E; G')\cap V_1'|)^t\mathbf{1}\{E \subseteq PE_Q(G')\}\right]
&\leq \E_{\bm X_1}\left[\omega_\theta(|N(\bm F_E; G)|)^t\mathbf{1}\{\bm X_1 \subseteq N(\bm F_E;G)\}\right]\\
&\leq \left(\frac{\theta}{|N(\bm F_E;G)|}\right)^t\left(\frac{|N(\bm F_E;G)|}{|V_1|}\right)^t\\
& = \left(\frac{\theta}{|V_1|}\right)^t.
\end{align*}

For each $Q\in(V_{-1}')^d$ clearly $PE_Q(G')\subseteq PE_Q(G)$. Furthermore since no vertex of $G'$ is isolated we also have $\emptyset\neq PE_Q(G')$. Thus for every $Q\in V_{-1}^d$,
\begin{align*}
\E\left[\omega_{\theta}(Q,V_1';G')^t\mathbf1\{Q\in(V_{-1}')^d\}\right]
&\leq \sum_{\emptyset\neq E\subseteq PE_Q(G)}\E\left[\omega_\theta(|N(E; G')\cap V_1'|)^t\mathbf{1}\{E = PE_Q(G')\}\right] \\
&\leq \sum_{\emptyset\neq E\subseteq PE_Q(G)}\E\left[\omega_\theta(|N(E; G')\cap V_1'|)^t\mathbf{1}\{E \subseteq PE_Q(G')\}\right] \\
& \leq \sum_{\emptyset\neq E\subseteq PE_Q(G)}\left(\frac{\theta}{|V_1|}\right)^t\\
&\leq 2^{2^d}\left(\frac{\theta}{|V_1|}\right)^t.
\end{align*}

Finally, taking the sum over all possible choices of $Q$ yields the inequality
\[\mb E\left[\sum_{Q\in (V_{-1}')^d}\omega_\theta(Q, V_1'; G')^t\right]\le 2^{2^d}\left(\frac{\theta}{|V_1|}\right)^t v(G)^d.\qedhere\]
\end{proof}

For each $i\in [k]$ and $\mathcal{Q}\subseteq (V_{-i})^d$, we define the average defect
\begin{equation}
\label{eqn:def-mu}
\mu_{\theta, t}\left(\mc Q, V_i; G\right) = \frac{1}{|\mc Q|}\sum_{Q\in \mc Q}\omega_\theta(Q, V_i; G)^t.
\end{equation}

\cref{lem:simul-pruning} can be used to construct a subhypergraph for which we have control over the average defects.

\begin{corollary}
\label{cor:simul-pruning}
Suppose that $G$ is a $k$-partite $k$-uniform hypergraph on $V_1\sqcup \cdots \sqcup V_k$ where $|V_1| = \cdots = |V_k| = n$ with at least $pn^k$ edges. Let $d, t$ be positive integers satisfying $d\geq k-1$ and let $\theta>0$. There exists a subhypergraph $G'$ of $G$, with vertex partition $V_1' \sqcup \ldots \sqcup V_k'$, that contains at least $\tfrac12p^{(t+1)^k}n^k$ edges and for each $i\in [k]$, we have
\[\mu_{\theta, t}\left((V'_{-i})^d, V_i'; G'\right) \leq 2^{2^d+1}k^{d+1}p^{-\frac{d}{k-1}(t+1)^k}\left(\frac{\theta}{n}\right)^t.\]
\end{corollary}

\begin{proof}
Let $G' \subseteq G$ be a random subhypergraph defined by the random process in \cref{lem:simul-pruning}, with parts $V_1'\sqcup \cdots \sqcup V_k'$. We shall define
\[S_i = \sum_{Q\in (V_{-i}')^d}\omega_\theta(Q, V_i'; G')^t.\]
Then by \cref{lem:simul-pruning} we have $\mb E[S_i] \leq 2^{2^d}k^d\theta^t n^{d-t}$.  Since $d\geq k-1$, by convexity we know that
\[\mb E\left[e(G')^{\frac d{k-1}}\right]\ge \left(\mb E[e(G')]\right)^{\frac d{k-1}}\geq p^{\frac{d}{k-1}(t+1)^k}n^{\frac{dk}{k-1}}.\]
Therefore, there exists a choice of subhypergraph $G'$ for which
\[\frac{e(G')^{\frac d{k-1}}}{p^{\frac{d}{k-1}(t+1)^k}n^{\frac{dk}{k-1}}} - \sum_{i=1}^k \frac{1}{2k}\frac{S_i}{2^{2^d}k^d\theta^t n^{d-t}} \geq \mb E\left[\frac{e(G')^{\frac d{k-1}}}{p^{\frac{d}{k-1}(t+1)^k}n^{\frac{dk}{k-1}}} - \sum_{i=1}^k \frac{1}{2k}\frac{S_i}{2^{2^d}k^d\theta^t n^{d-t}}\right] \ge \frac{1}{2}.\]
This means that $e(G') \geq \tfrac12 p^{(t+1)^k}n^k$ and for each $i\in[k]$ we have 
\[\frac{S_i}{e(G')^{\frac{d}{k-1}}} \leq 2k2^{2^d}k^dp^{-\frac{d}{k-1}(t+1)^k}n^{-\frac{d}{k-1}}\left(\frac{\theta}{n}\right)^t.\]
Note that $e(G') \le \prod_{j=1}^k |V_j'| \leq n|V_{-i}'|^{k-1}$. We conclude that
\[\mu_{\theta, t}\left((V_{-i}')^d, V_i'; G'\right) = \frac{1}{|V'_{-i}|^d}S_i \le \frac{S_i n^{\frac{d}{k-1}}}{e(G')^{\frac{d}{k-1}}} < 2^{2^d+1}k^{d+1}p^{-\frac{d}{k-1}(t+1)^k}\left(\frac{\theta}{n}\right)^t.\qedhere\]
\end{proof}

\subsection{Embedding}
\label{sec:embedding}

In this section $G$ and $H$ always refer to hypergraphs of the following type.

\begin{setup}
\label{stp}
Fix positive integers $K\geq k\geq 2$ and $d\geq 1$. Let $G$ be a $k$-uniform hypergraph with vertex set partitioned as $V(G)=V_1\sqcup \cdots\sqcup V_K$. Assume that $G$ has no isolated vertices.

Let $H$ be a $k$-uniform hypergraph with vertex set partitioned as $V(H)=W_1\sqcup\cdots\sqcup W_K$. Assume that each edge of $H$ has vertices coming from $k$ different parts of the partition. For each vertex $v\in W_i$ with $i<K$, let $N^+(v)$ be the set of neighbors of $v$ in the 1-skeleton of $H$ which lie in $W_{i+1}\sqcup\cdots\sqcup W_K$. Assume that $|N^+(v)|\leq d$ for each vertex $v\in W_i$ with $i<K$. Pick $d$-tuples $f_v\in(W_{i+1}\sqcup\cdots\sqcup W_K)^d$ which consist of the elements of $N^+(v)$ in some order padded by $d-|N^+(v)|$ additional vertices.

Let $\theta_1,\ldots,\theta_{K-1}>0$ be real numbers. Finally assume that $|V_K|\geq |W_k|$.
\end{setup}

We use the following random greedy process based on \cite[Section 4]{Lee17} to define a random map $\psi\colon V(H)\to V(G)$.

\begin{algorithm}
\label{alg}
Suppose we are in \cref{stp}. Define a random map $\psi\colon V(H)\to V(G)$ as follows:
\begin{enumerate}[1.]
    \item  Order the vertices of $W_K$ arbitrarily as $x_1,\ldots, x_{|W_K|}$. For each vertex $x_j$, let $\psi(x_j)$ be a uniform random element of $V_K\setminus\{\psi(x_1),\ldots,\psi(x_{j-1})\}$.
    \item In decreasing order, for each $i<K$, given a map $\psi$ defined on $W_{i+1}\sqcup\cdots\sqcup W_K$, we extend $\psi$ to $W_i$ as follows. Let $x_1,\ldots, x_{|W_i|}$ be an ordering of $W_i$ such that $\omega_{\theta_i}(\psi(f_{x_j}),V_i;G)$ is non-increasing in $j$.
    \item After embedding $x_1, \dots, x_{j-1}$, we now embed $x_j$. Define
    \[N_j=N(PE_{\psi(f_{x_j})}(G);G)\cap V_i\] and \[L_j=N_j\setminus\{\psi(x_1),\ldots,\psi(x_{j-1})\}.\]
    \begin{enumerate}
        \item If $|N_j|=0$, let $\psi(x_j)$ be a uniform random vertex in $V_i$;
        \item If $|L_j|<|N_j|/2$, let $\psi(x_j)$ be a uniform random vertex in $N_j$;
        \item If $|L_j|\geq|N_j|/2$, let $\psi(x_j)$ be a uniform random vertex in $L_j$.
    \end{enumerate}
\end{enumerate}
\end{algorithm}

Our goal is to show that if $G$ has small average defect then \cref{alg} produces an embedding $\psi\colon V(H)\to V(G)$. In particular, we want to show that we always use step 3(c) and never steps 3(a) and 3(b).

\begin{notation}\label{not:algorithm analysis}
We now introduce some notation used to analyze \cref{alg}. For $i<K$, $x\in W_i$, and $\phi\colon V(H)\to V(G)$ we use the following shorthand:
\begin{enumerate}
\item $\theta_x=\theta_i$ and $V_x=V_i$;
\item $\omega(x;\phi)=\omega_{\theta_x}(\phi(f_x),V_x;G)$;
\item $\mathcal{Q}_x=V_{i_1}\times\cdots\times V_{i_{d}}$ where $f_x\in W_{i_1}\times\cdots\times W_{i_{d}}$;
\item $\mu_t(x)=\mu_{\theta_x,t}(\mathcal{Q}_x,V_x;G)$ and $\mu_t=\max_{x\in V(H)}\mu_t(x)$;
\item $\gamma=\max\{1,\max_{i<K}|V_i|/\theta_i\}$.
\end{enumerate}
We extend this notation to $i=K$ as follows. For $x\in W_K$ we write $V_x=V_K$ and $\omega(x;\phi)=0$.
\end{notation}

\begin{remark}
Note that $\mu_{\theta,t}(\mathcal{Q},V_i;G)$ is only finite for some choices of $\mathcal Q=V_{i_1}\times\cdots\times V_{i_d}$. For example, if there is an edge $e=\{v_1,\ldots,v_k\}$ with $v_j\in V_{i_j}$ say, then $\mu_{\theta,t}(\mathcal{Q},V_i;G)$ is infinite. This is because $e$ is an entire edge so it can have no extension to $V_i$. In our application we will choose the vectors $f_x$ carefully so that $\mu_t(x)=\mu_{\theta_x,t}(\mathcal{Q}_x,V_x;G)$ is always finite (and in fact quite small).
\end{remark}

\begin{lemma}[{cf. \cite[Lemma 4.3]{Lee17}}]
\label{lem:omega-embed}
Suppose we are in \cref{stp}. Assume that $\theta_i\geq 2|W_i|$ for all $i<K$. Let $\psi$ be the random map produced by \cref{alg}. For $s\geq 1$, let $\phi\colon V(H)\to V(G)$ be a deterministic map satisfying $\Pr(\psi=\phi)>0$ and $\sum_{x\in W_i}\omega(x;\phi)^s\leq\tfrac12\theta_i$ for all $i<K$. Then $\phi$ is an embedding of $H$ in $G$.
\end{lemma}

\begin{proof}
It suffices to prove this when $s=1$ since the hypothesis is stronger for larger $s$. We now prove that for each $i<K$, if we condition on $\psi|_{W_{i+1}\cup\cdots\cup W_K} = \phi|_{W_{i+1}\cup\cdots\cup W_K}$ then \cref{alg} embeds $W_i$ using step 3(c) for every vertex. 

Since $\omega(x_1;\phi)\geq \cdots\geq \omega(x_j;\phi)$, by hypothesis we have $\omega(x_j;\phi)\leq\tfrac{1}{2j}\theta_i$. We claim that $|N_j| \geq 2j > 0$. Indeed, by definition, $\omega(x_j;\phi)=\omega_{\theta_i}(|N_j|)$. Hence, if $|N_j| \in(0, \theta_i)$ then by definition of $\omega_{\theta_i}$ we have the bound $\theta_i/|N_j|\leq \theta_i/(2j)$. This implies that $|N_j|\geq 2j>0$. Additionally, we must have $|N_j| > 0$, for otherwise by definition $\omega(x_j;\phi) = \infty > \frac{1}{2}\theta_i$. Finally, since we have assumed that $\theta_i \geq 2|W_i| \geq 2j$, if $|N_j| \geq \theta_i$ then $|N_j| \geq 2j$, as claimed. Now, since $|L_j|\geq|N_j|-j+1 \geq |N_j|/2$ we are in case 3(c) as desired. In particular, this implies that $\phi\colon V(H)\to V(G)$ is injective. Furthermore, it implies that for every $x \in W_1\sqcup\cdots\sqcup W_{K-1}$ it is the case that $\phi(x) \in N(PE_{\phi(f_x)}(G);G)\cap V_i$.

Next we verify that $\phi$ is a hypergraph homomorphism. We will prove inductively that for every $e \in E^*(H)$ that is contained in $W_i \cup \cdots \cup W_K$ it holds that $\phi(e) \in E^*(G)$. For the base case, suppose that $e \subseteq W_K$ and $e \in E^*(H)$. Since every edge in $H$ has at most one vertex in each part we conclude that $e=\{x\}$ for some $x \in W_K$. Hence, since $G$ has no isolated vertices we have $\{\phi(x)\}\in E^*(G)$.

Assume that we have proved the claim for $e \subseteq W_{i+1}\cup\cdots\cup W_K$. Let $e \in E^*(H)$ with $e \subseteq W_i \cup \cdots \cup W_K$ and $e \cap W_i \neq \emptyset$. Since $e$ has at most one vertex from each part there is some $x \in W_i$ such that $\{x\} = e \cap W_i$. Set $e' = e \setminus \{x\}$ and note that $e' \in E^*(H)$. We have already shown that $\phi(x)$ is an element of $N(PE_{\phi(f_x)}(G);G)\cap V_i$. Furthermore, we know that $e'$ is a subset of $f_x$, and by the inductive assumption we know that $\phi(e')\in E^*(G)$. Combining these facts we see that $\phi(e)\in E^*(G)$.

We have shown that for every edge $e\in E(H)$, we have $\phi(e)\in E^*(G)$. Since $\phi$ is an injection $|\phi(e)|=|e|$. Since both $G$ and $H$ are $k$-uniform this implies that $\phi(e)\in E(G)$, completing the proof that $\phi$ is an embedding.
\end{proof}

In analyzing \cref{alg} the main difficulty is controlling the sums $\sum_{x\in W_i}\omega(x;\phi)$. If $\phi\colon  W_{i+1}\cup\cdots\cup W_{K}\to V_{i+1}\cup\cdots\cup V_{K}$ were a uniform random map, then in expectation this sum would simply be $|W_i|$ times the average defect. Thus our goal is to control how much \cref{alg} differs from a uniform random map.

Following \cite[Section 4]{Lee17} we make the following definitions. Denote by $\nu$ the probability distribution on maps $\psi\colon V(H)\to V(G)$ induced by \cref{alg}. For $I\subseteq V(H)$, the probability distribution obtained from $\nu$ by \emph{neutralizing $I$} is the distribution induced by following the randomized process above but each time we reach a vertex $x\in I$ we choose the image of $x$ to be a uniform random vertex of $V_x$.

\begin{lemma}[{cf. \cite[Lemma 4.4]{Lee17}}]
\label{lem:neutralizing}
Suppose we are in \cref{stp}. Assume that $|V_K|\geq 2|W_K|$. Let $\nu$ be the probability distribution on maps $\psi\colon V(H)\to V(G)$ induced by \cref{alg}. For sets $I_1,I_2,J\subseteq V(H)$ with $I_2=I_1\cup J$, let $\nu_i$ be the distribution obtained from $\nu$ by neutralizing $I_i$ for $i=1,2$. Let $X$ be any random variable which only depends on the images of the vertices in $J$. If $t=|J\setminus I_1|\geq 1$, then
\[\E_{\nu_1}[X]\leq 2^t\gamma^t\E_{\nu_2}[X]+2^{2t-1}\gamma^{2t}\E_{\nu_2}[X^2]+\frac1{2t}\sum_{y\in J\setminus I_1}\E_{\psi\sim\nu_2}[\omega(y;\psi)^{2t}].\]
\end{lemma}

Other than a slight change in the definition of $C_y(\phi)$, this proof is identical to \cite[Proof of Lemma 4.4]{Lee17}.

\begin{proof}
For $y\in I_2\setminus I_1$, define
\[C_y(\phi)=\frac{2|V_y|}{|N(PE_{\phi(f_y)}(G);G)\cap V_y|}\] if $y\notin W_K$ and $C_y(\phi)=2$ if $y\in W_K$.  The motivation for this definition (as well as the assumption that $|V_K|\geq 2|W_K|$) is that when embedding $y$ in \cref{alg}, each particular vertex is chosen as its image with probability at most $C_y(\phi)/|V_y|$. Furthermore, we have the inequality $C_y(\phi)\leq 2\gamma \max\{1, \omega(y;\phi)\}$ for all $y$. 

For $0\leq s\leq v(H)$, define $\psi_s$ to be the random partial map $V(H)\to V(G)$ defined by running \cref{alg} until it embeds $s$ vertices. For a set $W\subseteq V(H)$ of size $s$ and a deterministic map $\phi\colon W\to V(G)$ our goal is to understand the event $\psi_s=\phi$. Note that this event occurs with nonzero probability only for some choices of $W$ and $\phi$.

Define $I_W= I_2\setminus(W\cup I_1)$ and let $\nu_W$ be the distribution obtained from $\nu$ by neutralizing $I_1\cup I_W$. Our goal is to prove
\begin{equation}
\label{eq:neutralize-induction}
\E_{\psi\sim \nu_1}[X|\psi_s=\phi]\leq\E_{\psi\sim \nu_W}\left[X\cdot\prod_{y\in I_W}C_y(\psi)\right|\left.\vphantom{\prod_{I_W}}\psi_s=\phi\right]
\end{equation}
for all choices of $s,W,\phi$ with $s=|W|$. We abuse notation and say that \cref{eq:neutralize-induction} holds vacuously for those $s,W,\phi$ for which $\Pr(\psi_s=\phi)=0$. 

We prove \cref{eq:neutralize-induction} by induction on $s$. For the base case, suppose $s= v(H)$. Then $I_W=\emptyset$ so $\nu_1=\nu_W$ and the product disappears. 

Let $W$ be a set of size $s$ and suppose that \cref{eq:neutralize-induction} holds for all larger sets. Given $\psi_s=\phi$, we know which vertex will be embedded next by the random process $\nu_1$. Let this vertex be $z\in W_i$ for some $i$. For each vertex $v\in V_i$, write $\phi_v\colon W\cup\{z\}\to V(G)$ for the map which sends $z$ to $v$ and otherwise agrees with $\phi$. Then we have
\begin{equation}\label{eq:X conditioned on psi_t}
\begin{split}
\E_{\psi\sim \nu_1}[X|\psi_s=\phi]
&=\sum_{v\in V_i}\Pr_{\nu_1}(\psi(z)=v|\psi_s=\phi)\E_{\psi\sim \nu_1}[X|\psi_{s+1}=\phi_v]\\
&\leq \sum_{v\in V_i}\Pr_{\nu_1}(\psi(z)=v|\psi_s=\phi)\E_{\psi\sim \nu_{W\cup\{z\}}}\left[X\cdot\prod_{y\in I_{W\cup\{z\}}}C_y(\psi)\right|\left.\vphantom{\prod_{I_{W\cup\{z\}}}}\psi_{s+1}=\phi_v\right]\\
&= \sum_{v\in V_i}\Pr_{\nu_1}(\psi(z)=v|\psi_s=\phi)\E_{\psi\sim \nu_W}\left[X\cdot\prod_{y\in I_{W\cup\{z\}}}C_y(\psi)\right|\left.\vphantom{\prod_{I_{W\cup\{z\}}}}\psi_{s+1}=\phi_v\right].
\end{split}
\end{equation}
To see the last line, note that $\nu_W$ and $\nu_{W\cup \{z\}}$ only differ on whether or not $z$ is neutralized. However, conditioning on $\psi_{s+1}=\phi_v$, we see that the image of $z$ is already fixed and thus the two distributions become identical. Let us point out that we continue to abuse notation slightly in the above equation since for some choices of $v$ the event $\psi_{s+1}=\phi_v$ never occurs so we are conditioning on a null event. However for these terms we are multiplying by 0 so the equation continues to be valid.

We have two cases. Either $I_{W\cup\{z\}}=I_W$, or $I_{W\cup\{z\}} \cup \{z\} =I_W$. If $I_{W\cup\{z\}}=I_W$, then $z\notin I_W$ so $\Pr_{\nu_1}(\psi(z)=v|\psi_s=\phi)=\Pr_{\nu_W}(\psi(z)=v|\psi_s=\phi)$ (to see this, note that we are conditioning on the fact that the two processes agree for the first $s$ steps and since neither of them neutralize $z$ they agree on step $s+1$ as well). Thus in this case, the final expression in \eqref{eq:X conditioned on psi_t} can be written as 
\[\sum_{v\in V_i}\Pr_{\nu_W}(\psi(z)=v|\psi_s=\phi)\E_{\psi\sim \nu_W}\left[X\cdot\prod_{y\in I_W}C_y(\psi)\right|\left.\vphantom{\prod_{I_W}}\psi_{s+1}=\phi_v\right]=\E_{\psi\sim \nu_W}\left[X\cdot\prod_{y\in I_{W}}C_y(\psi)\right|\left.\vphantom{\prod_{I_W}}\psi_s=\phi\right].\]

In the second case $I_{W\cup\{z\}} \cup \{z\} =I_W$. This means that $z\in I_W$, but $z\notin I_1$. We claim that this implies that
\[\Pr_{\nu_1}(\psi(z)=v|\psi_s=\phi)\leq C_z(\phi)\Pr_{\nu_W}(\psi(z)=v|\psi_s=\phi).\]
Indeed, since $z$ is neutralized in $\nu_W$, each element of $V_i$ gets chosen as $\psi(z)$ with probability $1/|V_i|$ under $\nu_W$. Since $z$ is not neutralized in $\nu_1$, by our earlier observation each vertex gets chosen as $\psi(z)$ with probability at most $C_z(\phi)/|V_i|$ under $\nu_1$ conditioned on $\psi_s=\phi$. This implies the desired inequality. Thus in this case, we can write
\begin{align*}
\E_{\psi\sim \nu_1}[X|\psi_s=\phi]
&\leq \sum_{v\in V_i}C_z(\phi)\Pr_{\nu_W}(\psi(z)=v|\psi_s=\phi)\E_{\psi\sim \nu_W}\left[X\cdot\prod_{y\in I_{W\cup\{z\}}}C_y(\psi)\right|\left.\vphantom{\prod_{I_{W\cup\{z\}}}}\psi_{s+1}=\phi_v\right]\\
&= \sum_{v\in V_i}\Pr_{\nu_W}(\psi(z)=v|\psi_s=\phi)\E_{\psi\sim \nu_W}\left[X\cdot\prod_{y\in I_{W}}C_y(\psi)\right|\left.\vphantom{\prod_{I_{W}}}\psi_{s+1}=\phi_v\right]\\
&=\sum_{v\in V_i}\E_{\psi\sim \nu_W}\left[X\cdot\prod_{y\in I_{W}}C_y(\psi)\right|\left.\vphantom{\prod_{I_{W}}}\psi_{s}=\phi\right].
\end{align*}
Thus we have completed the induction in either case.

Now we apply \cref{eq:neutralize-induction} with $s=0$. This gives the bound
\[\E_{\psi\sim\nu_1}[X]\leq\E_{\psi\sim\nu_2}\left[X\cdot\prod_{y\in  I_2\setminus I_1}C_y(\psi)\right].\]
Since $|N(PE_{\psi(f_y)}(G);G)\cap V_y|\geq\min\{\theta_y,\theta_y/\omega(y;\psi)\}$ we have the bound \[C_y(\psi)\leq 2\tfrac{|V_y|}{\theta_y}\max\{1,\omega(y;\psi)\}\leq 2\gamma\max\{1,\omega(y;\psi)\}\] for all $y\notin V_K$. Note that the bound also trivially holds for $y\in V_K$.

Recalling that $I_2\setminus I_1=J\setminus I_1$ is a set of size $t\geq 1$, we conclude that 
\[\prod_{y\in  J\setminus I_1}C_y(\psi)\leq2^t\gamma^t\prod_{y\in  J\setminus I_1}\max\{1,\omega(y;\psi)\}\leq 2^t\gamma^t\frac1t\sum_{y\in  J\setminus I_1}(1+\omega(y;\psi)^t).\]

Finally this gives
\begin{align*}
\E_{\psi\sim\nu_1}[X]
&\leq 2^t\gamma^t\E_{\psi\sim\nu_2}[X]+\frac1t\sum_{y\in  J\setminus I_1}\E_{\psi\sim\nu_2}[2^t\gamma^tX\cdot\omega(y;\psi)^t]\\
&\leq 2^t\gamma^t\E_{\psi\sim\nu_2}[X]+\frac1t\sum_{y\in  J\setminus I_1}\E_{\psi\sim\nu_2}\left[\frac12\left(2^{2t}\gamma^{2t}X^2+\omega(y;\psi)^{2t}\right)\right]\\
&= 2^t\gamma^t\E_{\psi\sim\nu_2}[X]+2^{2t-1}\gamma^{2t}\E_{\psi\sim\nu_2}[X^2]+\frac1{2t}\sum_{y\in  J\setminus I_1}\E_{\psi\sim\nu_2}\left[\omega(y;\psi)^{2t}\right].\qedhere
\end{align*}
\end{proof}

We now iterate this neutralizing procedure. The following proof is very similar to \cite[Proof of Lemma 4.5]{Lee17} applied to $H^{(1)}$.

\begin{lemma}[{cf. \cite[Lemma 4.5]{Lee17}}]
\label{lem:random-defect-bound}
Suppose we are in \cref{stp}. Assume that $|V_K|\geq 2|W_K|$. Let $\nu$ be the probability distribution on maps $\psi\colon V(H)\to V(G)$ induced by \cref{alg}. For all vertices $x\in V(H)$,
\[\E_{\psi\sim\nu}[\omega(x;\psi)^{2d}]\leq 2^{2d+1}\gamma^{2d}\mu_{4d}.\]
\end{lemma}

\begin{proof}
Fix $x\in V(H)$. We construct a sequence of rooted trees $T_0\subseteq T_1\subseteq \cdots\subseteq T_m$ whose vertices are labelled by elements of $V(H)$. Let $T_0$ be the tree consisting of the single vertex $a$ labelled $x$. To construct $T_{i+1}$ from $T_i$ we do the following procedure. For each leaf $b$ of $T_i$, let $A(b)$ be the set of ancestors of $b$ in $T_i$, and let $I(b)=\bigcup_{c\in A(b)}f_c$. In this proof we use $f_c$ to refer to the set $f_y$ where $c\in T_i$ is labelled by the vertex $y\in V(H)$.  Then we add as children of $b$ new vertices labelled $f_b\setminus I(b)$. Call this set of added vertices $C(b)$.

Note that our trees may have multiple vertices with the same label. However, the process eventually terminates. Indeed, any path from the root contains distinct vertices (except that there may be a second vertex labelled $x$ on the path). Therefore the depth is at most $v(H)$. Call the final tree $T_m$.

We define a vertex-weighting $\sigma\colon V(T_m)\to[0,1]$ as follows. Define $\sigma(a)=1$ and whenever we add a new vertex which is a child of $b$, assign it the weight $\sigma(b)/(2|f_b\setminus I(b)|)$. Note that total weight of the new leaves in $T_i$ (that is, $V(T_i)\setminus V(T_{i-1})$) is $2^{-i}$, so the total weight of the vertices in any of the trees is less than 2.

Let $\nu_b$ be the probability distribution formed from $\nu$ by neutralizing $I(b)$. For notational convenience let $T_{-1}$ be the empty tree. We wish to prove that for all $i\geq 0$ the following holds
\begin{equation}
\label{eq:tree-neutral}
\E_{\psi\sim\nu}[\omega(a;\psi)^{2d}]\leq 2^{2d}\gamma^{2d}\mu_{4d}\sum_{b\in V(T_{i-1})} \sigma(b)+\sum_{b\in V(T_i)\setminus V(T_{i-1})}\sigma(b)\E_{\psi\sim \nu_b}[\omega(b;\psi)^{2d}].
\end{equation}
We will prove \cref{eq:tree-neutral} by induction on $i$. The base case $i=0$ is trivial since $I(a)=\emptyset$ so $\nu=\nu_a$. 

For $b\in V(T_{i+1})\setminus V(T_i)$, let $\tilde\nu_b$ be the probability distribution obtained from $\nu$ by neutralizing $I(b)\cup f_b$. Write $t_b=|f_b\setminus I(b)|$. By assumption $H$ has skeletal degeneracy at most $d$ so $t_b\leq d$. If $t_b\geq 1$ then by \cref{lem:neutralizing} (with $J=f_b$) we have
\begin{align*}
\E_{\psi\sim \nu_b}[\omega(b;\psi)^{2d}]
&\leq 2^{t_b}\gamma^{t_b}\E_{\psi\sim \tilde \nu_b}[\omega(b;\psi)^{2d}]+2^{2t_b-1}\gamma^{2t_b}\E_{\psi\sim\tilde \nu_b}[\omega(b;\psi)^{4d}]+\frac1{2t_b}\sum_{c\in f_b\setminus I(b)}\E_{\psi\sim \tilde\nu_b}[\omega(c;\psi)^{2d}]\\
&\leq 2^{2d}\gamma^{2d}\E_{\psi\sim\tilde \nu_b}[\omega(b;\psi)^{4d}]+\frac1{2t_b}\sum_{c\in f_b\setminus I(b)}\E_{\psi\sim \tilde\nu_b}[\omega(c;\psi)^{2d}].
\end{align*}
Since $f_b$ is neutralized by $\tilde\nu_b$ we have that the expectation that $\E_{\psi\sim\tilde\nu_b}[\omega(b;\psi)^{4d}]$ agrees with $\mu_{4d}$, defined in \eqref{eqn:def-mu}. Furthermore, it follows from the definitions that for each $c\in f_b\setminus I(b)$ we have $\tilde\nu_b=\nu_c$. Thus we have shown that 
\[\E_{\psi\sim \nu_b}[\omega(b;\psi)^{2d}]\leq 2^{2d}\gamma^{2d}\mu_{4d}+\frac1{2t_b}\sum_{c\in f_b\setminus I(b)}\E_{\psi\sim \nu_c}[\omega(c;\psi)^{2d}]
\]
whenever $t_b>0$.

If $t_b=0$, then $f_b\subseteq I(b)$. This means that all of $f_b$ is already neutralized by $\nu_b$ and so $\E_{\psi\sim\tilde\nu_b}[\omega(b;\psi)^{2d}]=\mu_{2d}\leq\mu_{4d}$. We abuse notation and use the above bound in either case and simply ignore the second term when $t_b=0$. 

Putting together these results with the inductive hypothesis, we see that 
\begin{align*}
\E_{\psi\sim\nu}[\omega(a;\psi)^{2d}]
&\leq 2^{2d}\gamma^{2d}\mu_{4d}\sum_{b\in V(T_{i-1})} \sigma(b)+\sum_{b\in V(T_i)\setminus V(T_{i-1})}\sigma(b)\E_{\psi\sim \nu_b}[\omega(b;\psi)^{2d}]\\
&\leq 2^{2d}\gamma^{2d}\mu_{4d}\sum_{b\in V(T_{i-1})} \sigma(b)\\&\qquad\qquad+\sum_{b\in V(T_i)\setminus V(T_{i-1})}\sigma(b)\left(2^{2d}\gamma^{2d}\mu_{4d}+\frac1{2t_b}\sum_{c\in f_b\setminus I(b)}\E_{\psi\sim\nu_c}[\omega(c;\psi)^{2d}]\right)\\
&= 2^{2d}\gamma^{2d}\mu_{4d}\sum_{b\in V(T_{i})} \sigma(b)+\sum_{c\in V(T_{i+1})\setminus V(T_i)} \sigma(c)\E_{\psi\sim\nu_c}[\omega(c;\psi)^{2d}].
\end{align*}
This completes the induction. Finally, \cref{eq:tree-neutral} with $i=m+1$ gives the desired result.
\end{proof}

\begin{theorem}[{cf. \cite[Theorem 4.2]{Lee17}}]
\label{thm:embedding}
Suppose we are in \cref{stp}. Assume that $|V_K|\geq 2|W_K|$, that $\theta_i\geq 2|W_i|$ for all $i<K$ and that $2^{2d+2}\gamma^{2d}\mu_{4d}\sum_{i<K}|W_i|/\theta_i<1$. Then there is an embedding of $H$ into $G$.
\end{theorem}

\begin{proof}
By \cref{lem:random-defect-bound} for each $i<K$,
\[\E_{\psi\sim \nu}\left[\sum_{x\in W_i}\omega(x;\psi)^{2d}\right]\leq |W_i|2^{2d+1}\gamma^{2d}\mu_{4d}.\]
Thus by Markov's inequality
\[\Pr_{\psi\sim \nu}\left(\sum_{x\in W_i}\omega(x;\psi)^{2d}>\frac12\theta_i\right)<\frac{|W_i|}{\theta_i}2^{2d+2}\gamma^{2d}\mu_{4d}.\]

Therefore we see that by hypothesis, there is a deterministic map $\phi$ that satisfies $\Pr(\psi=\phi)>0$ and $\sum_{x\in W_i}\omega(x;\phi)^{2d}\leq\tfrac12\theta_i$ for all $i\in[K]$. By \cref{lem:omega-embed}, $\phi$ is an embedding of $H$ into $G$.
\end{proof}

\subsection{Putting it all together}
\label{sec:proof-of-main}

\subsubsection{Partitioning $H$}

Say that two vertices of a hypergraph $H$ are neighbors if they are adjacent in the 1-skeleton $H^{(1)}$.

\begin{lemma}
\label{lem:H-partition}
Let $H$ be an $n$-vertex, $k$-uniform, $k$-partite hypergraph with $\skel{H}\leq d$ and vertex partition $V(H)=W_1\sqcup \cdots \sqcup W_k$. There exists a refinement of the partition $\{W_i^{(j)}\}_{i\in [T], j\in [k]}$, with the following properties:
\begin{enumerate}[(i)]
\item $T\leq \log_2 n$;
    \item for all $i\in [T]$, $j\in [k]$, we have $|W_i^{(j)}|\leq 2^{-i+1}n$;
    \item for all $j\in [k]$, the set $\bigcup_{i\in [T]}W_i^{(j)} = W_j$; and
    \item for all $i\in [T]$, $j\in [k]$, each vertex in $W_i^{(j)}$ has at most $4d$ neighbors in $\bigcup_{i'\geq i,j'\in [k]}W_{i'}^{(j')}$.
\end{enumerate}
\end{lemma}

\begin{remark}
\cref{lem:H-partition} follows by applying \cite[Lemma 3.1]{Lee17} to $H^{(1)}$. For completeness we include the proof here as well.
\end{remark}

\begin{proof}
Let $U_1 = V(H)$, and recursively define $U_{i+1}\subseteq U_i$ to be the set of vertices having degree at least $4d$ in the induced subgraph $H^{(1)}[U_i]$. Finally define $V_i=U_i\setminus U_{i+1}$ and $W_i^{(j)} = V_i\cap W_j$.

By assumption $H^{(1)}$ is $d$-degenerate. Hence there are at most $d|U_i|$ edges in $H^{(1)}[U_i]$, so $4d|U_{i+1}|\le 2d|U_i|$. Equivalently, $2|U_{i+1}|\le |U_i|$. Therefore $U_{T+1} = \emptyset$ for some $T\le \log_2 n$, which gives (i). Because $2|U_{i+1}|\le |U_i|$, we have $|W_i^{(j)}|\le |U_i|\le 2^{-i+1}|U_1| = 2^{-j+1}n$, which gives (ii).

Meanwhile we have $\bigcup_{i\in [T]}V_i = U_1\setminus U_{T+1} = V(H)$ which immediately implies (iii).  Finally, $W_i^{(j)} \subseteq U_i\setminus U_{i+1}$, and $\bigcup_{i'\geq i, j'\in [k]}W_{i'}^{(j')} = U_i$, so all vertices in $W_i^{(j)}$ have degree less than $4d$ in $H^{(1)}[U_{i}]$, thereby establishing (iv) as desired.
\end{proof}

\subsubsection{Partitioning $G$}

In the proof of \cref{thm:lin-turan}, after applying \cref{cor:simul-pruning}, we will construct a subhypergraph of $G$ in which we have good control of the $\theta$-defect. \cref{lem:G-partition} asserts that such a subhypergraph exists. To prove \cref{lem:G-partition} we rely on the next lemma, which quantifies the way in which $\mu_{\theta,t}$ controls the worst-case behavior of the $\theta$-defect. (Note that $\mu_{\theta,t}$ controls the average-case behavior of the defect by definition.)

\begin{lemma}
\label{lem:lower-bound-defect}
Let $G$ be a $k$-uniform $k$-partite hypergraph on $V_1\sqcup \cdots \sqcup V_k$. For any $i, \theta, t$ and $\mc Q$ a collection of subsets of $V(G)\setminus V_i$, if $\mu_{\theta,t}(\mathcal{Q},V_i;G)>0$ then every tuple $Q\in \mc Q$ satisfies
\[|N(PE_Q(G); G) \cap V_i| \geq \frac{\theta}{(|\mc Q|\cdot \mu_{\theta, t}(\mc Q, V_i; G))^{1/t}}.\]
If $\mu_{\theta,t}(\mathcal{Q},V_i;G)=0$ then every tuple $Q\in\mathcal{Q}$ satisfies $|N(PE_Q(G); G) \cap V_i| \geq\theta$.
\end{lemma}

\begin{proof}
If $\mu_{\theta,t}(\mathcal{Q},V_i;G)=0$ then, by definition of $\mu_{\theta,t}$ and $\omega_\theta$ (see \eqref{eq:omega_theta set def}, \eqref{eq:omega_theta function def}, and \eqref{eqn:def-mu}) we have $|N(PE_Q(G); G) \cap V_i| \geq\theta$ for every $Q \in \mc Q$.

Henceforth, we assume that $\mu_{\theta,t}(\mathcal{Q},V_i;G)>0$. By \eqref{eq:omega_theta function def}, if $\omega_\theta(x)\neq 0$ then $\omega_\theta(x)>1$. Since $\mu_{\theta,t}(\mathcal{Q},V_i;G)>0$, by \eqref{eqn:def-mu} we actually have $|\mathcal{Q}|\cdot\mu_{\theta,t}(\mathcal{Q},V_i;G)\geq 1$.

Let $Q\in\mathcal{Q}$. Suppose, for a contradiction, that 
\[|N(PE_Q(G); G) \cap V_i|< \frac{\theta}{(|\mc Q|\cdot \mu_{\theta, t}(\mc Q, V_i; G))^{1/t}}<\theta\]
(with the last inequality following from the discussion above). Then by \cref{eqn:def-mu} we have
\[|\mathcal{Q}|\cdot\mu_{\theta,t}(\mathcal{Q},V_i;G)\geq\omega_\theta(|N(PE_Q(G); G) \cap V_i|)^t=\left(\frac{\theta}{|N(PE_Q(G); G) \cap V_i|}\right)^t.\]
This contradicts our assumption, proving the desired bound on the size of the common neighborhood of $Q$. 
\end{proof}

\begin{lemma}[{cf.\ \cite[Lemmas 5.3 and 6.2]{Lee17}}]
\label{lem:G-partition}
There exists an absolute constant $C$ such that the following holds. Fix any $T,d,t, k, m, \theta\in \mb N$ and $\epsilon, \epsilon'\in (0, 1)$ that satisfy $k\geq 2$, $t\geq 4d$, $m \geq \left(\frac{k^dTt}{\epsilon^d\epsilon'}\right)^{C}$ and $\theta\geq \epsilon m$. Fix constants $p_1,p_2,\ldots,p_T\geq m^{-1/(10d)}$ such that $\sum_i p_i\leq 1$. Let $\theta_i=p_i\theta/4$ for each $i\in [T]$. Let $G$ be a $k$-uniform $k$-partite hypergraph on $A_1\sqcup\cdots\sqcup A_k$. Suppose that $\epsilon m\leq |A_j|\leq m$ and $\mu_{\theta,t}(A_{-j}^d,A_j;G)<\tfrac12$ for each $j\in [T]$. Then there exist disjoint sets $V_i^{(j)}$ for $i\in[T]$ and $j\in[k]$ such that
\begin{enumerate}[(i)]
\item\label{itm:Vij cardinalities} $V_i^{(j)}\subseteq A_j$ and $\tfrac14p_i|A_j|\leq |V_i^{(j)}|\leq p_i|A_j|$ for all $i\in[T]$ and $j\in[k]$,
\item\label{itm:Vij average defect} for all $i\in[T]$, $j\in [k]$ and $i_1,\ldots,i_d\in[T]$, $j_{1},\ldots,j_{d}\in[k]\setminus\{j\}$,
\[\mu_{\theta_i,t}\left(\prod_{a=1}^{d}V_{i_a}^{(j_a)},V_i^{(j)};G\right)\leq \max\left\{\epsilon',8\epsilon^{-d}k^d\mu_{\theta,t}(A_{-j}^d,A_j;G)\right\}.\]
\end{enumerate}
\end{lemma}

We will prove the lemma by giving a randomized construction for the sets $V_i^{(j)}$ and then showing that they satisfy \cref{itm:Vij cardinalities,itm:Vij average defect} with positive probability. A na\"ive attempt at such a construction might be to take $V_i^{(j)}$ as a uniformly random subset of $A_j$ of size $p_i|A_j|$ (in such a way that the sets are disjoint) and then attempt to prove that \cref{itm:Vij average defect} holds with positive probability. Indeed, this is almost the approach we use. However, in order to prove \cref{itm:Vij average defect} we rely on McDiarmid's inequality. To apply this we would require the average defect to satisfy a bounded differences condition, which in fact fails in general. To rectify this, we precede the random construction by removing vertices that might violate our desired bounded differences condition.

We first state the versions of Hoeffding's inequality and McDiarmid's inequality that we use.

\begin{theorem}[{Hoeffding's inequality (see, e.g., \cite[Theorem 2.2.6]{Ver18}}]
\label{thm:hoeffding}
Let $X_1,\ldots,X_N$ be independent $\{0,1\}$-valued random variables. Then for any $t\geq 0$,
\[\Pr\left(\sum_{i=1}^N (X_i-\E[X_i])\geq t\right)\leq\exp\left(-\frac{2t^2}{N}\right).\]
\end{theorem}

\begin{theorem}[{McDiarmid's inequality (see, e.g., \cite[Theorem 2.9.1]{Ver18}}]
\label{thm:mcdiarmid}
Let $\cS$ be a finite set and let $X_1,\ldots,X_N$ be independent $\cS$-valued random variables. Let $f\colon \cS^N\to \mathbb{R}$ be a function such that $|f(X)-f(X')|\leq c$ where $X,X'\in \cS^N$ are arbitrary $N$-tuples that agree on $N-1$ coordinates. Then for any $t\geq 0$,
\[\Pr\left(f(X_1,\ldots,X_n)-\E[f(X_1,\ldots,X_n)]\geq t\right)\leq\exp\left(-\frac{2t^2}{c^2N}\right).\]
\end{theorem}

\begin{proof}[Proof of \cref{lem:G-partition}]
For each $j\in[k]$ we define $B_j\subseteq A_j$ by removing the vertices that contribute too much to any $\mu_{\theta,t}(A_{-i}^d,A_i;G)$. Specifically, for each $i\in [k]$ let $R_i\subseteq A_{-i}$ be the set of vertices $v\in A_{-i}$ such that
\[\sum_{Q:v\in Q\in A_{-i}^d}\omega_\theta (Q,A_i;G)^t\geq|A_{-i}|^{d-5/8}.\]
By definition (see\ \eqref{eqn:def-mu})
\[
\mu_{\theta,t}(A_{-i}^d,A_i;G) = \frac{1}{d|A_{-i}^d|}\sum_{v \in A_{-i}^d} \sum_{{Q \in A_{-i}^d : v \in Q}}\omega_\theta (Q,A_i;G)^t.
\]
Hence
\[|R_i|\cdot|A_{-i}|^{d-5/8}\leq d|A_{-i}|^d\mu_{\theta,t}(A_{-i}^d,A_i;G),\]
so, applying the lemma's assumptions, $|R_i|\leq dkm^{5/8}$ for all $i\in [k]$.

Next, for each $j\in[k]$ define $B_j=A_j\setminus \bigcup_{i\neq j}R_i$. Fix some $j\in[k]$. By the bound above we have $|B_j|\geq |A_j|-dk^2m^{5/8}$. By our assumption on $m$ (and taking $C$ sufficiently large) we have $m^{3/8}\geq \frac{t^2k^2}{\epsilon} \geq\frac{dk^2}{\epsilon}(1-2^{-\frac{1}{2d}})^{-1}$. Using the additional assumption that $|A_j| \geq \varepsilon m$ we conclude that
\begin{equation}\label{eq:B_j lower bound}
|B_j| \geq |A_j| \left( 1 - dk^2 \frac{m^{5/8}}{|A_j|} \right) \geq  |A_j| \left( 1 - \frac{dk^2}{\varepsilon m^{3/8}} \right) \geq 2^{-1/(2d)}|A_j|.
\end{equation}
Similarly, for every $Q \in A_{-j}^d$ we have $|N(PE_Q(G);G)\cap B_j|\geq |N(PE_Q(G);G)\cap A_j|-dk^2m^{5/8}$. By \cref{lem:lower-bound-defect}, the assumption $\mu_{\theta,t}(A^d_{-j},A_j;G)< 1/2$, and the assumption $t\geq 4d$ we have $|N(PE_Q(G);G)\cap A_j|\geq\theta|A_{-j}|^{-d/t}\geq \theta|A_{-j}|^{-1/4} \geq \epsilon k^{-1/4} m^{3/4}$. (If $\mu_{\theta,t}(A^d_{-j},A_j;G)=0$ then we have the stronger bound $|N(PE_Q(G);G)\cap A_j|\geq\theta$.) Again by our assumption on $m$, taking $C$ sufficiently large, we have $m^{1/8} \geq \frac{t^2k^{9/4}}{\epsilon} \geq \frac{dk^{9/4}}{\epsilon}(1-2^{- 1/(2t)})^{-1}$ which implies that $|N(PE_Q(G);G)\cap B_j|\geq 2^{-1/(2t)}|N(PE_Q(G);G)\cap A_j|$. Thus
\begin{equation}\label{eq:mu B_j upper bound}
\begin{split}
\mu_{\theta,t}(B_{-j}^d,B_j;G)
&=\frac1{|B_{-j}|^d}\sum_{Q\in B_{-j}^d}\omega_\theta(|N(PE_Q(G);G)\cap B_j|)^t\\
&\leq \frac{2^{1/2}}{|B_{-j}|^d}\sum_{Q\in B_{-j}^d}\omega_\theta(|N(PE_Q(G);G)\cap A_j|)^t\\
&\leq \frac{2}{|A_{-j}|^d}\sum_{Q\in B_{-j}^d}\omega_\theta(|N(PE_Q(G);G)\cap A_j|)^t\\
&\leq 2\mu_{\theta,t}(A_{-j}^d,A_j;G).
\end{split}
\end{equation}

We also obtain the lower bound
\begin{equation}\label{eq:NB_j lower bound}
|N(PE_Q(G);G)\cap B_j| \geq 2^{-1/(2t)} \epsilon k^{-1/4}m^{3/4} \geq \varepsilon k^{-1/4}m^{3/4}\!/2.
\end{equation}

Finally, for every vertex $v \in B_{-j}$, applying similar reasoning and the definition of $B_{-j}$ we have
\begin{equation}\label{eq:omega sum over v}
\sum_{Q:v\in Q\in B_{-j}^d}\omega_\theta (Q,B_j;G)^t\leq2^{1/2}\sum_{Q:v\in Q\in A_{-j}^d}\omega_\theta (Q,A_j;G)^t\leq2^{1/2}|A_{-j}|^{d-5/8}\leq 2|B_{-j}|^{d-5/8}.
\end{equation}

We now give a randomized construction for the desired sets $V_i^{(j)}$, and show that \cref{itm:Vij cardinalities,itm:Vij average defect} are satisfied with positive probability. For each $j \in [k], i \in [T]$, and vertex $v \in B_j$ we assign $v$ to $V_i^{(j)}$ with probability $q_i = p_i/2$ (so that $v$ is assigned to at most one set and to no set with probability $1-\sum_i q_i$). We make these assignments independently for each vertex. In this way we obtain the sets $\{V_i^{(j)}\}_{j\in[k],i\in[T]}$ which by construction are disjoint and also satisfy $V_i^{(j)} \subseteq A_j$ for every $i$ and $j$.

Let $E_1$ be the event that for all $i\in[T]$, for all $j\in[k]$, and for all $Q\in B_{-j}^d$ we have
\[|N(PE_Q(G);G)\cap V_i^{(j)}|\geq \frac{q_i}2|N(PE_Q(G);G)\cap B_j|.\]
Let $E_2$ be the event that for all $i\in[T]$, $j\in [k]$ and $i_1,\ldots,i_d\in[T]$, $j_{1},\ldots,j_{d}\in[k]\setminus\{i\}$,
\[\mu_{\theta,t}\left(\prod_{a=1}^{d}V_{i_a}^{(j_a)},B_j;G\right)\leq \max\left\{\epsilon',4\epsilon^{-d}k^d\mu_{\theta,t}(B_{-j}^d,B_j;G)\right\}.\]
Let $E_3$ be the event that for all $i\in[T]$, $j\in [k]$ we have
\[\frac{q_i}{2^{1/(2d)}}|B_j|\leq|V_i^{(j)}|\leq 2q_i|B_j|.\]

We will prove that with positive probability, all of $E_1,E_2,E_3$ occur. We will then show that $E_1 \land E_2 \land E_3$ implies \cref{itm:Vij cardinalities,itm:Vij average defect}.

We begin with $E_3$. Fix $i\in [T]$ and $j\in [k]$. By Hoeffding's inequality (\cref{thm:hoeffding}),
\begin{align*}
\Pr\left(|V_i^{(j)}|<\frac{q_i}{2^{1/(2d)}}|B_j|\right)
&\leq \exp \left( -\frac{2(q_i|B_j|(1-2^{-1/(2d)}))^2}{m} \right),\\
\Pr\left(|V_i^{(j)}|>2q_i|B_j|\right)
&\leq \exp \left( - \frac{2(q_i|B_j|)^2}{m} \right).
\end{align*}
Using the assumption that $p_i \geq m^{-1/(10d)}$, the inequality $1-2^{-1/(2x)} \geq 1/(4x)$ (which holds for all $x >0$), and \eqref{eq:B_j lower bound} we obtain the inequality  
\[q_i|B_j|\geq \left(1-2^{-1/(2d)}\right)q_i|B_j|\geq \left(\frac1{4d}\right)\left(\frac12 m^{-1/(10d)}\right)\left(2^{-1/(2d)}\epsilon m\right)\geq \frac{\epsilon m^{9/10}}{16d}.\]
Union-bounding over all choices of $i,j$, we see that
\[\Pr(\neg E_3)\leq 2Tk \exp \left(- \frac{\epsilon^2 m^{4/5}}{128d^2} \right).\]
By our assumption on $m$ and taking $C$ sufficiently large we have $m^{4/5}\geq \frac{t^{6}Tk}{\epsilon^2} \geq \frac{4^{6} d^2Tk}{\epsilon^2}$, so
\[
2Tk \exp \left(- \frac{\epsilon^2 m^{4/5}}{128d^2} \right) \le 2Tke^{-32Tk} < 1/4.
\]

We turn out attention to $E_1$, for which the analysis is similar. Fix $i\in[T]$, $j\in [k]$, and $Q\in B_{-j}^d$. By Hoeffding's inequality (\cref{thm:hoeffding}),
\[\Pr\left(|N(PE_Q(G);G)\cap V_i^{(j)}|< \frac{q_i}2|N(PE_Q(G);G)\cap B_j|\right)\leq \exp \left( - \frac{2(q_i|N(PE_Q(G);G)\cap B_j|/2)^2}{m} \right).\]
Applying \eqref{eq:NB_j lower bound} we obtain
\[
q_i|N(PE_Q(G);G)\cap B_j| \geq \tfrac12 m^{-1/(10d)} \varepsilon k^{-1/4}m^{3/4}\!/2 \geq \varepsilon k^{-1/4}m^{13/20}\!/4.
\]
Union-bounding over all choices of $i,j,Q$, we see that
\[\Pr(\neg E_1)\leq Tk(km)^d \exp \left( - \frac{\varepsilon^2 m^{13/10}}{32k^{1/2}m} \right) = Tk(km)^d \exp \left( - \frac{\varepsilon^2 m^{3/10}}{32k^{1/2}} \right).\]
Repeatedly applying the assumption on $m$, as long as $C$ is sufficiently large we have $4Tk(km)^d \leq m^{2d}$ and $m^{3/20}\geq \log m$ as well as $m^{3/20}\geq 64dk^{1/2} / \varepsilon^2$. Putting these together we obtain $\Pr(\neg E_1) \le 1/4$.

Finally we consider $E_2$. First we study the related event $E_2'$, defined to be the event that for all $i\in[T]$, $j\in [k]$ and $i_1,\ldots,i_d\in[T]$, $j_{1},\ldots,j_{d}\in[k]\setminus\{j\}$, writing $\mathcal Q=\prod_{a=1}^dV_{i_a}^{(j_a)}$ and $\mathcal B=\prod_{a=1}^d B_{j_a}$, we have
\[|\mathcal Q|\mu_{\theta,t}\left(\mathcal Q,B_j;G\right)\leq  |\mathcal B|\prod_{a=1}^d q_{i_a}\cdot \max\left\{2\mu_{\theta,t}(\mathcal B,B_j;G),\frac{\epsilon'}{2}\right\}.\]

Fix some $i\in[T]$, $j\in [k]$ and $i_1,\ldots,i_d\in[T]$, $j_{1},\ldots,j_{d}\in[k]\setminus\{j\}$.
Let $\mathcal Q,\mathcal B$ be as defined above. Then
\begin{align*}
\E\left[|\mathcal Q|\mu_{\theta,t}\left(\mathcal Q,B_j;G\right)\right]
&=\E\left[\sum_{Q\in\mathcal Q}\omega_{\theta}(|N(PE_Q(G);G)\cap B_j|)^t\right]\\
&=\sum_{Q\in \mathcal B}\omega_{\theta}(|N(PE_Q(G);G)\cap B_j|)^t\cdot\Pr(Q\in \mathcal Q)
\end{align*}
Let $\partial \mathcal B$ be the subset of $\mathcal B$ that consists of $d$-tuples with at least one repeated coordinate. For $Q\in\mathcal B\setminus\partial\mathcal B$ we have $\Pr(Q\in\mathcal Q)=\prod_{a=1}^d q_{i_a}$. Thus we see
\begin{align*}
\E\left[|\mathcal Q|\mu_{\theta,t}\left(\mathcal Q,B_j;G\right)\right]
&\leq |\mathcal B|\prod_{a=1}^d q_{i_a}\cdot\mu_{\theta,t}(\mathcal B,B_j;G)+\sum_{Q\in \partial\mathcal B}\omega_{\theta}(|N(PE_Q(G);G)\cap B_j|)^t.
\end{align*}
Observe that for $Q\subseteq Q'$ it holds that $\omega_{\theta}(|N(PE_Q(G);G)\cap B_j|)\leq \omega_{\theta}(|N(PE_{Q'}(G);G)\cap B_j|)$. Thus, for every $\ell \in [d]$, we have
\[\sum_{Q\in \prod_{a\in[d]\setminus\{\ell\}}B_{j_a}}\omega_{\theta}(|N(PE_Q(G);G)\cap B_j|)^t\leq \frac1{|B_{j_\ell}|}\sum_{Q\in \mathcal B}\omega_{\theta}(|N(PE_Q(G);G)\cap B_j|)^t.\]

Using the bound $|B_{j_\ell}|\geq 2^{-1/(2d)}\epsilon m$ (which follows from \eqref{eq:B_j lower bound} and the assumption $|A_{j_\ell}| \geq \varepsilon m$) and summing over all $d(d-1)/2$ choices of where the repeated coordinates occur, we have the bound
\begin{align*}
\sum_{Q\in \partial\mathcal B}\omega_{\theta}(|N(PE_Q(G);G)\cap B_j|)^t
&\leq\frac{d(d-1)}{2\cdot 2^{-1/(2d)}\epsilon m}\sum_{Q\in \mathcal B} \omega_\theta(|N(PE_Q(G);G)\cap B_j|)^t\\
&\leq \frac{d(d-1)}{2\cdot 2^{-1/(2d)}\epsilon m} |B_{-j}|^d\mu_{\theta,t}(B_{-j}^d,B_j;G)\\
&\leq \frac{d(d-1)}{2\cdot 2^{-1/(2d)}\epsilon m} k^dm^d
\end{align*}
(where the final inequality follows from \eqref{eq:mu B_j upper bound} and the assumption $\mu_{\theta,t}(A_{-j}^d,A_j;G) \leq 1/2$). Thus
\[\E[|\mathcal Q|\mu_{\theta,t}\left(\mathcal Q,B_j;G\right)]\leq |\mathcal B|\prod_{a=1}^d q_{i_a}\cdot\mu_{\theta,t}(\mathcal B,B_j;G)+\epsilon^{-1} d^2 k^dm^{d-1}. \]

Observe that \eqref{eq:omega sum over v} and the bound $|B_{-j}|\leq km$ imply that the random variable $|\mathcal Q|\mu_{\theta,t}\left(\mathcal Q,B_j;G\right)$ only changes by at most $2k^dm^{d-5/8}$ when the assignment of a single vertex $v$ to a part $V_i^{(j)}$ is changed. Thus by McDiarmid's inequality (\cref{thm:mcdiarmid}) we see that
\begin{equation*}
\label{eq:azuma-e2}
\Pr\left(|\mathcal Q|\mu_{\theta,t}\left(\mathcal Q,B_j;G\right)\geq  |\mathcal B|\prod_{a=1}^d q_{i_a}\cdot \max\left\{2\mu_{\theta,t}(\mathcal B,B_j;G),\frac{\epsilon'}{2}\right\}\right)\leq \exp \left(-\frac{2\Delta^2}{m(2k^dm^{d-5/8})^2} \right)
\end{equation*}
where
\begin{align*}
\Delta
&=|\mathcal B|\prod_{a=1}^d q_{i_a}\cdot \max\left\{2\mu_{\theta,t}(\mathcal B,B_j;G),\frac{\epsilon'}2\right\}-\left(|\mathcal B|\prod_{a=1}^d q_{i_a}\cdot\mu_{\theta,t}(\mathcal B,B_j;G)+\epsilon^{-1} d^2 k^dm^{d-1}\right)\\
&\geq |\mathcal B|\prod_{a=1}^d q_{i_a}\cdot\max\left\{\mu_{\theta,t}(\mathcal B,B_j;G),\frac{\epsilon'}{4}\right\}-\epsilon^{-1}d^2k^dm^{d-1}.
\end{align*}
The last line follows from the easy inequality $\max\{x,y\}-x/2\geq\max\{x/2,y/2\}$.

Note that $\prod_{a=1}^d q_{i_a}\geq(\tfrac12 m^{-1/(10d)})^d=2^{-d}m^{-1/10}$. Also, $|\mathcal B|\geq (2^{-1/(2d)}\epsilon m)^d= 2^{-1/2}\epsilon^d m^d$. We bound the maximum from below by $\epsilon'/4$, giving
\[\Delta\geq 2^{-d-5/2}\epsilon^d\epsilon'm^{d-1/10}-\epsilon^{-1}d^2k^dm^{d-1}\geq 2^{-4d}\epsilon^d\epsilon' m^{d-1/10},\]
where the last inequality follows from our assumed lower bound on $m$ and provided that $C$ is sufficiently large.

Union-bounding over all choices of $i,j,i_{1},\ldots,i_d,j_1,\ldots,j_d$ we see that 
\[\Pr(\neg E_2')\leq (Tk)^{d+1} \exp\left( -\frac{2(2^{-4d}\epsilon^d\epsilon' m^{d-1/10})^2}{4k^{2d}m^{2d-1/4}} \right) = (Tk)^{d+1} \exp\left(- \frac{\epsilon^{2d}\epsilon'^2}{2^{8d+1} k^{2d}} m^{1/20} \right).\]
Assuming that $C$ is sufficiently large, we have $m^{1/20} \geq \frac{2^{10d}k^{2d}}{\epsilon^{2d}\epsilon'^2}(d+1)\log(Tkt)$. Then $\Pr(\neg E_2')\le t^{-d-1} < 1/4$.

It now follows that $\Pr(E_1 \land E_2' \land E_3) \geq 1 - \Pr(\neg E_1) - \Pr(\neg E_2') - \Pr(\neg E_3) \geq 1/4$. Thus, with positive probability all of $E_1,E_2',E_3$ occur. When this happens we have that
\begin{align*}
\mu_{\theta,t}\left(\prod_{a=1}^{d}V_{i_a}^{(j_a)},B_j;G\right)
& \leq \frac{|\mathcal B|}{|\mathcal Q|}\prod_{a=1}^d q_{i_a}\max\left\{2\mu_{\theta,t}(\mathcal B,B_j;G),\tfrac{\epsilon'}{2}\right\}\\
&\leq \max\left\{2^{3/2}\mu_{\theta,t}(\mathcal B,B_j;G),\epsilon'\right\},
\end{align*}
where the first inequality follows from $E_2'$ and the second from the fact that $E_3$ implies $|\mc B| \prod_{a=1}^d q_{i_a} \leq 2^{1/2}|\mathcal Q|$. Finally 
\[\mu_{\theta,t}(\mathcal B,B_j;G)\leq \frac{|B_{-j}|^d}{|\mathcal B|}\mu_{\theta,t}(B_{-j}^d,B_j;G)\leq \left(\frac{km}{2^{-1/(2d)}\epsilon m}\right)^d\mu_{\theta,t}(B_{-j}^d,B_j;G),\] so when $E_1,E_2',E_3$ occur, $E_2$ also occurs.

To finish the proof we show that when $E_1 \land E_2 \land E_3$ occurs the desired conclusions are satisfied. By $E_3$ we have $|V_i^{(j)}|\leq 2q_i|B_j|\leq 2q_i|A_j|=p_i|A_j|$ and
\[
|V_i^{(j)}| \geq \frac{q_i}{2^{1/(2d)}}|B_j| \geq \frac{q_i}{2^{1/d}}|A_j|\geq \tfrac14 p_i|A_j|
\]
where the middle inequality follows from \eqref{eq:B_j lower bound}. Thus \cref{itm:Vij cardinalities} holds. Furthermore, for all $i,j,Q$ by $E_1$ we have
\[|N(PE_Q(G);G)\cap V_i^{(j)}|\geq \frac{q_i}2|N(PE_Q(G);G)\cap B_j|\]
so since $\theta_i=q_i\theta/2$ we have $\omega_{\theta_i}(Q,V_i^{(j)};G)\leq\omega_{\theta}(Q,B_j;G)$. Thus 
\begin{align*}
\mu_{\theta_i,t}\left(\prod_{a=1}^{d}V_{i_a}^{(j_a)},V_i^{(j)};G\right)
&\leq \mu_{\theta,t}\left(\prod_{a=1}^{d}V_{i_a}^{(j_a)},B_j;G\right)\\
&\leq \max\left\{\epsilon',4\epsilon^{-d}k^d\mu_{\theta,t}(B_{-j}^d,B_j;G)\right\}\tag{since $E_2$ occurs}\\
&\leq \max\left\{\epsilon',8\epsilon^{-d}k^d\mu_{\theta,t}(A_{-j}^d,A_j;G)\right\}.
\end{align*}
Therefore we have shown with positive probability the sets $V_i^{(j)}$ satisfy the desired properties.
\end{proof}

\subsubsection{Finishing up}
\begin{proof}[Proof of \cref{thm:lin-turan}]
Let $H$ be a $k$-partite $k$-uniform hypergraph on $h$ vertices with $\skel{H}=d$. Since $\skel{H}>0$ we conclude that $H$ is nonempty, which implies that $d\geq k-1$. Let $V(H)=W_1\sqcup\cdots\sqcup W_k$ be the vertex partition of $H$. We apply \cref{lem:H-partition} to further partition $H$. This produces $W_i^{(j)}$ for $i\in[T]$, $j\in[k]$ with $T\leq\log_2 h$. For each $j\in[k]$ such that $W_T^{(j)}$ is empty we modify $H$ by adding an isolated vertex to $H$ that is in this vertex set. We abuse notation and also refer to the modified hypergraph as $H$.

Define
\[\eta=\frac{p^{\frac{4d}{k-1}(16d+1)^k}}{2^{2^{4d}+1}k^{4d+1}}\qquad\text{and}\qquad\theta=\eta^3n.\]
Let $G$ be a $k$-partite $k$-uniform hypergraph with parts consisting of $n$ vertices each and with at least $pn^k$ edges. By \cref{cor:simul-pruning} (applied with $t=16d$, and noting that $4d \geq k-1$) we can find a subhypergraph $G'$ of $G$ with vertex parts $A_1\sqcup\cdots\sqcup A_k$ with at least $\tfrac12 p^{(16d+1)^k}n^k$ edges such that
\[\mu_{\theta,16d}(A_{-j}^{4d},A_j;G')\leq 2^{2^{4d}+1}k^{4d+1}p^{-\frac{4d}{k-1}(16d+1)^k}\left(\frac{\theta}{n}\right)^{16d}\leq \eta^{47d}.\]
for each $j\in [k]$.

We apply \cref{lem:G-partition} with parameters $(T,d,t,k,m,\theta,\varepsilon,\varepsilon') = (T, 4d, 16d, k, n, \theta, \eta^3, \eta^{34d}$). We quickly verify that the assumptions of \cref{lem:G-partition} hold. Clearly, we have $k\geq 2$ and $16d\geq4\cdot 4d$. Let $C$ be the universal constant from \cref{lem:G-partition}. By taking $c_{p,d,k}$ sufficiently small and noting that \cref{thm:lin-turan} holds vacuously when $c_{p,k,d}n<1$, we may assume that $n$ is sufficiently large that $n \geq \left(\frac{k^{4d}T\cdot 16d}{(\eta^3)^d\eta^{34d}}\right)^{C}$ holds. By definition $\theta\geq\eta^3n$.

For each $i\in [T]$, let $p_i=c2^{-i/(80d)}$ where $c$ is the normalizing constant such that $\sum_i p_i=1$. We have
\[c=\frac1{\sum_{i=1}^T 2^{-i/(80d)}}\geq \frac1{160d}.\]
Then $\sum p_i\leq 1$ and $p_i\geq p_T\geq 2^{-\log_2 n/(80d)}/(160d)\geq n^{-1/(40d)}$ (since we can assume that $n > (160d)^{80d}$).

Since $G'$ has at least $\tfrac12 p^{(16d+1)^k}n^k$ edges, we see that $\eta^3n\leq \tfrac12 p^{(16d+1)^k}n\leq |A_j|\leq n$ for each $j\in[k]$. Finally, we know that $\mu_{\theta,16d}(A_{-j}^{4d},A_j;G)\leq \eta^{47d}<\tfrac12$.

Having verified its assumptions, \cref{lem:G-partition} now produces disjoint sets $V_i^{(j)}$ for $i\in[T]$, $j\in[k]$ such that $\tfrac14 p_i|A_j|\leq |V_i^{(j)}|\leq p_i|A_j|\leq p_i n$ and 
\[\mu_{\theta_i,16d}\left(\prod_{a=1}^{4d}V_{i_a}^{(j_a)},V_i^{(j)};G\right)\leq \max\left\{\eta^{34d},8\eta^{-12d}k^{4d}\mu_{\theta,16d}(A_{-j}^{4d},A_j;G)\right\}=\eta^{34d}\]
for all $i\in[T]$, $j\in [k]$ and $i_1,\ldots,i_{4d}\in[T]$, $j_{1},\ldots,j_{4d}\in[k]\setminus\{j\}$. In the above expression, $\theta_i=p_i\theta/4$.

Now we wish to apply \cref{thm:embedding} to show that \cref{alg} produces an embedding of $H$ into $G$ with positive probability. We have produced partitions $V_i^{(j)}$ and $W_i^{(j)}$ of $V(G)$ and $V(H)$ respectively into $K=Tk\leq k\log_2 h$ parts. We order these parts by the lexicographic order on $(i,j)$. Note that in particular, each edge of $H$ has vertices coming from $k$ different parts of this partition, and each vertex $v\in V(H)$ has at most $4d$ neighbors in later parts of this partition. Finally, for \cref{stp} we need to pick a $4d$-tuple $f_v$ for each $v\in W_i^{(j)}$ with $(i,j)\neq (T,k)$. We produce this tuple by arbitrarily ordering the neighbors $N^+(v)$ and then padding with elements of $W_{i'}^{(j')}$ where $j'\neq j$ and $(i,j)$ precedes $(i',j')$ in lexicographic order. (We can always do this padding because the $W_T^{(j')}$ are never empty). 

In what follows we use the notation from \cref{not:algorithm analysis}. By our choice of $f_x$, we have that $\mc Q_x$ is a set of the form that is controlled by \cref{lem:G-partition}(ii). Thus we know $\mu_{16d}(x)=\mu_{\theta_x,16d}(\mc Q_x; V_x)\leq \eta^{34d}$ and so $\mu_{16d}\leq \eta^{34d}$.

Now we check the hypotheses of \cref{thm:embedding}. By our choice of $\theta$ we have
\[\theta_j=p_j\theta/4\geq  \frac1{160d} 2^{-j/(80d)} \eta^3n\geq 2^{-j+2}h\geq 2|W_i^{(j)}|.\]
The second inequality holds for a sufficiently-small choice of $c_{p,d,k}>0$ and the assumption $h\leq c_{p,d,k}n$.
Similarly, we have
\[|V_T^{(k)}|\geq \frac14 p_T|A_k|\geq\frac1{640d} 2^{-T/(80d)} \eta^3 n\geq 2^{-T+2}h\geq 2|W_T^{(k)}|.\]

Now
\[\gamma\leq \max_{i,j}\frac{|V_i^{(j)}|}{\theta_j}\leq \max_{j}\frac{p_jn}{p_j\theta/4}\leq\frac{4n}{\theta}=\frac4{\eta^3}.\]
Similarly,
\[\frac{|W_i^{(j)}|}{\theta_j}\leq \frac{2^{-j+1}h}{p_j\theta/4}\leq\frac{2^{-j+1}h}{\frac1{160d}2^{-j/(80d)}\eta^3n}\leq \frac{327680h}{2^{j/2} \eta^3 n}\leq\frac1{2^{j/2}}\]
where we have again used the fact that $h\leq c_{p,d,k}n$ for a sufficiently-small choice of $c_{p,d,k}>0$.

Therefore 
\begin{align*}
2^{8d+2}\gamma^{8d}\mu_{16d}\sum_{i,j}|W_i^{(j)}|/\theta_i
& \leq 2^{8d+2}\left(\frac{4}{\eta^3}\right)^{8d}\eta^{34d}\sum_{i,j}\frac1{2^{j/2}}\\
&\leq k2^{24d+4}\eta^{10d}<1.
\end{align*}
Thus, \cref{thm:embedding} implies that \cref{alg} produces an embedding of $H$ into $G$ with positive probability, as desired.
\end{proof}

\begin{remark}
The above proof works for $c_{p,d,k}\leq (k^{-1}\eta)^{100Cd}$. For fixed $d,k$ this gives $c_{p,d,k}\leq p^{2^{O(k \log d)}}$.
\end{remark}

\section{Concluding remarks}
\label{sec:open}

Define $\exp^{(k)}(x)$ (the exponential tower of height $k$ with $x$ at the top) by $\exp^{(1)}(x)=2^x$ and $\exp^{(k)}(x)=2^{\exp^{(k-1)}(x)}$. For $k\geq 4$, our proof of \cref{thm:lin-ramsey-main} gives $r(H;q)\leq\exp^{(k-1)}(O_{k,q}(d))n$. This is known to be best-possible for most regimes of the parameter choices. For example, for $k\geq 3$ and $q\geq 4$ it is known \cite{EHR65} that the complete $k$-uniform hypergraph has $r(K_{d+1}^{(k)};q)\leq\exp^{(k-1)}(O_{k,q}(d))$. For $q=2,3$ it is not known if there is a matching lower bound but this is generally conjectured to be the case (see \cite{CFS10} for the best-known lower bound). Thus for $k\geq 4$ we should not expect to be able to improve our bound on $r(H;q)$ in general. As further evidence for this being the correct $d$-dependence note that it also matches the best-known upper bound for bounded-degree hypergraphs \cite{CFS09}.

However, for $k=3$ our argument only gives the bound $r(H;q)\leq\exp^{(2)}(O_q(d\log d))n$. In fact, with a little more care our argument gives $r(H;q)\leq \exp^{(2)}(O_q(d))n$ for 3-uniform $H$ as long as $H^{(1)}$ has bounded chromatic number, but this does not improve the bound for general 3-uniform $d$-skeletal degenerate hypergraphs. We conjecture that the $\log d$ factor can be removed from the top of the tower.

\begin{conjecture}
\label{conj:3-uniform-double-exp}
For a 3-uniform $n$-vertex hypergraph $H$ with $d_1(H)=d$, \[r(H;q)\leq \exp^{(2)}(O_q(d))n.\]
\end{conjecture}

A similar statement \cite{BFS23} is known when the skeletal degeneracy is replaced by the square root of the number of edges. \cref{conj:3-uniform-double-exp} is stronger than this result since $d_1(H)\leq O(\sqrt{e(H^{(1)})})\leq O(\sqrt{e(H)})$.

\subsection{Better bounds on the Tur\'an number}

\cref{thm:k-unif-turan-main} states that for every $k$-uniform $k$-partite hypergraph $H$,
\begin{equation}\label{eqn:ex(n,H) rates of growth}
\Omega_k\left(n^{k-\frac{C_k}{d_1(H)}}\right) \leq \ex(n,H) \leq O_H\left(n^{k-\frac{c_k}{d_1(H)^{k-1}}}\right).
\end{equation}
We would like to better understand the gap between the lower and upper bounds. We will momentarily suggest a finer measure of degeneracy and conjecture that it better captures the exponent of $\ex(n,H)$. Before doing so we point out that both rates of growth in \eqref{eqn:ex(n,H) rates of growth} are possible.

First note that \cref{prop:complete-k-unif-k-part-lower} implies that the upper bound in \eqref{eqn:ex(n,H) rates of growth} is tight for the complete $k$-uniform $k$-partite hypergraph $K^{(k)}_{d,\ldots,d}$ (up to the constant $c_k$). Furthermore, \cref{prop:bipartite-hedgehog-upper} shows show that the lower bound in \eqref{eqn:ex(n,H) rates of growth} is tight for the bipartite hedgehog $H_d^{(k)}$ (up to the constant $C_k$).

Given these two examples, it is clear that $d_1(H)$ is not a fine enough measure to capture the Tur\'an exponent closely. Instead we consider $d_{\mathsf{max}}(H) = \max_{1\leq i<k} d_i(H)$, the maximum of the $i$-th skeletal degeneracy over all $i$ (see \cref{defn:skeletal degeneracy} for the definition).

\cref{thm:k-uniform-construction-ith-degen} implies that there is some constant $C_k>0$ such that for every $k$-uniform $k$-partite hypergraph $H$ with at least two edges, we have\[\ex(n,H)\geq \Omega_k(n^{k-C_k/d_{\mathsf{max}}(H)}).\]

We conjecture that this result is tight up to the constant $C_k$.

\begin{conjecture}
\label{conj:d-max-turan}
There is a constant $c_k>0$ such that for every $k$-uniform $k$-partite hypergraph $H$,
\[\mathrm{ex}(n,H)\leq O_H\left(n^{k-\frac{c_k}{d_{\mathsf{max}}(H)}}\right).\]
\end{conjecture}

\subsection{Relations between various skeletal degeneracies}

For a $k$-uniform $H$, note the inequalities $d_i(H)\leq \binom{d_1(H)} {i}$, implying that 
\begin{equation*}
\skel{H}\leq d_{\mathsf{max}}(H)\leq \frac{\skel{H}^{k-1}}{(k-1)!}.
\end{equation*}
(In fact one can deduce more relations between the $d_i(H)$ using the Kruskal--Katona theorem.)

It follows from standard properties of the usual notion of degeneracy that for any $k$-uniform $H$ and any $1\leq i<k$ there is an ordering of $V(H)$ such that each vertex is the last vertex of at most $d_i(H)$ edges of $H^{(i)}$. However, at first sight it seems possible that this produces $k-1$ different orderings, one for each value of $i$. We show the nonobvious fact that (losing a factor of $k^2$) one can take a single ordering simultaneously for all $i$. This is a property of max-degeneracy that may be useful in attacking \cref{conj:d-max-turan}. 

\begin{proposition}
For every $k$-uniform hypergraph $H$ there is an ordering $v_1,v_2,\ldots,v_n$ of $V(H)$ such that for each vertex $v_i$, the number of sets $S\subseteq\{v_1,\ldots,v_{i-1}\}$ such that there exists an edge $e$ of $H$ with $e\cap\{v_1,\ldots,v_i\}=S\cup\{v_i\}$ is at most $k^2 d_{\mathsf{max}}(H)$.
\end{proposition}

\begin{proof}
For a $k$-uniform hypergraph $H=(V,E)$ write $H^*[U]=(U,\{e\cap U:e\in E\})$ for the not-necessarily-uniform ``induced subhypergraph'' on $U$.

We will first show that for each nonempty $U\subseteq V$ there exists a vertex $v\in U$ such that
\begin{equation}
\label{eq:low-deg-vertex}
\sum_{r=1}^k \deg_r(v,H^*[U])\leq \sum_{i=0}^{k-1} (i+1)d_i(H).
\end{equation}
where $\deg_r(v,G)$ is the number of $r$-edges of $G$ which contain $v$.

Write $E_r(G)$ for the set of $r$-edges of $G$. Let $U \subseteq V$ be nonempty. For the sake of contradiction, suppose that \cref{eq:low-deg-vertex} fails for every $v\in U$. Then we have
\[\sum_{r=1}^k r|E_r(H^*[U])|=\sum_{r=1}^k\sum_{v\in U}\deg_r(v,H^*[U])>|U|\sum_{i=0}^{k-1} (i+1)d_i(H).\]

However, since the $(i+1)$-uniform part of $H^*[U]$ has $i$-th skeletal degeneracy at most $d_i(H)$, we know that $|E_{i+1}(H^*[U])|\leq |U|d_i(H)$. Summing over all $i$ contradicts the above inequality.

Now to produce the ordering, suppose that we have already chosen vertices $v_{i+1},\ldots,v_n$. Let $v_i$ be a vertex that satisfies \cref{eq:low-deg-vertex} for $U=V(H) \setminus \{ v_{i+1},\ldots,v_n \}$. Iterating this procedure $n$ times produces an ordering with the required property since $\sum_{i=0}^{k-1} (i+1)d_i(H)\leq k^2d_\mathsf{max}(H)$.
\end{proof}

\subsection{More conjectures on max-skeletal degeneracy}

As shown above, we have the bounds $d_1(H)\leq d_{\mathsf{max}}(H)\leq d_1(H)^{k-1}$. The hardest cases for \cref{conj:d-max-turan} should be for those $H$ where there is a large gap between $d_{\mathsf{max}}(H)$ and $d_1(H)^{k-1}$. An interesting example comes from Latin squares.

\begin{conjecture}
Let $L$ be a $d\times d$ Latin square. Define $H_L$ to be the 3-uniform 3-partite hypergraph on vertex set $[d]\sqcup[d]\sqcup[d]$ where $(i,j,k)\in E(H_L)$ if and only if $L(i,j)=k$. Then $\mathrm{ex}(n,H_L)=n^{3-\Theta(1/d)}$.
\end{conjecture}

Note that for every $d\times d$ Latin square we have $d_1(H_L),d_2(H_L)=\Theta(d)$, but \cref{thm:k-unif-turan-main} only shows that $\mathrm{ex}(n,H_L)\leq O(n^{3-c/d^2})$. It is worth pointing out that the 1-skeleton of $H_L$ is $K_{d,d,d}$ which is also the 1-skeleton of $K^{(3)}_{d,d,d}$, the complete 3-partite 3-uniform hypergraph. Indeed, our techniques only show that a hypergraph with $n^{3-c/d^2}$ edges contains $H_L$ because it in fact contains $K^{(3)}_{d,d,d}$ and therefore \textit{every} 3-partite $H$ with at most $d$ vertices in each side.

It would still be interesting to find some Latin square in any hypergraph of this density (not necessarily some specific one), or even some large subset of a Latin square. Along these lines we make the following weaker conjecture that we are still unable to prove.

\begin{conjecture}
There is a constant $c>0$ such that for all $d\geq 4$ the following holds. If $G$ is a (not necessarily linear) 3-uniform hypergraph on $n$ vertices with at least $n^{3-c/d}$ edges, then there is a linear subhypergraph of $G$ with $d$ vertices and at least $d^2/100$ edges.
\end{conjecture}

Note that if the linearity constraint on the subhypergraph is omitted then the conjecture is easily seen to be true. Indeed, it follows from a hypergraph generalization of the standard K\H{o}v\'ari--S\'os--Tur\'an argument (first written down by Erd\H{o}s \cite{Erd64}), that any $3$-uniform hypergraph with at least $n^{3-c/d}$ edges contains a copy of $K^{(3)}_{1,(d-1)/2,(d-1)/2}$. This is a (nonlinear) hypergraph with $d$ vertices and $(d-1)^2/4$ edges.

\end{document}